\documentclass[10pt]{amsart}
\usepackage{amsmath}
\usepackage{amsfonts,amssymb}
\usepackage{graphicx}
\usepackage{enumerate}
\usepackage{amsthm}
\usepackage[all]{xy}
\usepackage{extarrows}
\usepackage{color}
\usepackage{lmodern}
\usepackage{shuffle}
\usepackage{mathabx}
\usepackage{lmodern}
\usepackage{tikz}
\usepackage{hyperref}
\usepackage{scalefnt}
\usepackage{a4wide}
\usepackage{tikz}
\usetikzlibrary{matrix,arrows,decorations.pathmorphing}

\newcommand {\bR}{\mathbb R}
\newcommand {\bN}{\mathbb N}
\newcommand {\bZ}{\mathbb Z}
\newcommand {\bC}{\mathbb C}

\newcommand {\bQ}{\mathbb Q}

\newcommand {\Id}{\operatorname{Id}}

\newcommand {\be}{\mathbf{1}}

\newcommand{\cA}{\mathcal{A}}

\newcommand{\cH}{\mathcal{H}}
\newcommand{\cR}{\mathcal{R}}

\newcommand{\wt}{\operatorname{wt}}

\newcommand{\fh}{\mathfrak{h}}
\newcommand{\fH}{\mathfrak{H}}

\newcommand{\oP}{\overline{P}}

\newcommand {\z}{\zeta}

\newcommand{\sh}{\shuffle}
\newcommand{\shl}{\shuffle_\lambda}

\newcommand{\tsh}{\,\tilde{\shuffle}\,}
\newcommand{\ssh}{\shuffle_\star}

\newcommand{\oDelta}{\overline{\Delta}}
\newcommand{\tast}{\tilde{\ast}}
\long\def\ignore#1{}

\newtheorem{theorem}{Theorem}[section]
\newtheorem {lemma}[theorem]{Lemma}
\newtheorem {proposition}[theorem]{Proposition}

\newtheorem {corollary}[theorem]{Corollary}

\newtheorem {remark}[theorem]{Remark}

\begin{document}

\title[Duality and ($q$-)MZV\lowercase{s}]{Duality and ($q$-)multiple zeta values}

\author[K.~Ebrahimi-Fard]{Kurusch Ebrahimi-Fard}
\address{{ICMAT,C/Nicol\'as Cabrera, no.~13-15, 28049 Madrid, Spain}. {\tiny{On leave from UHA, Mulhouse, France.}}}
\email{kurusch@icmat.es, kurusch.ebrahimi-fard@uha.fr}
\urladdr{www.icmat.es/kurusch}

\author[D.~Manchon]{Dominique Manchon}
\address{Univ. Blaise Pascal, C.N.R.S.-UMR 6620, 3 place Vasar\'ely, CS 60026, 63178 Aubi\`ere, France}
\email{manchon@math.univ-bpclermont.fr}
\urladdr{http://math.univ-bpclermont.fr/~manchon/}

\author[J.~Singer]{Johannes Singer}
\address{Department Mathematik, Friedrich-Alexander-Universit\"at Erlangen-N\"urnberg, Cauerstra\ss e 11, 91058 Erlangen, Germany}
\email{singer@math.fau.de}
\urladdr{www.math.fau.de/singer}
\keywords{multiple zeta values, $q$-analogues, duality, double shuffle relations, derivation relations, Hoffman--Ohno relation, Hopf algebra, infinitesimal bialgebra, Rota--Baxter algebra}
\subjclass[2010]{11M32,16T05,05A30}

\date{\today}

\begin{abstract}
Following Bachmann's recent work on bi-brackets and multiple Eisenstein series, Zudilin introduced the notion of multiple $q$-zeta brackets, which provides a $q$-analog of multiple zeta values possessing both shuffle as well as quasi-shuffle relations. The corresponding products are related in terms of duality. In this work we study Zudilin's duality construction in the context of classical multiple zeta values as well as various $q$-analogs of multiple zeta values. Regarding the former we identify the derivation relation of order two with a Hoffman--Ohno type relation. Then we describe relations between the Ohno--Okuda--Zudilin $q$-multiple zeta values and the Schlesinger--Zudilin $q$-multiple zeta values. 
\end{abstract}

\maketitle

\tableofcontents


\section{Introduction}
\label{sect:intro}

\emph{Multiple zeta values} (MZVs) are nested sums of depth $n \in \mathbb{N}$ and weight $\sum_{i=1}^n k_i > n$, for positive integers $k_1 > 1,k_i>0$, $i=2,\ldots,n$
\begin{align}
\label{eq:defMZVs}
 \z(k_1,\ldots,k_n):=\sum_{m_1>\cdots >m_n>0}\frac{1}{m_1^{k_1}\cdots m_n^{k_n}}. 
\end{align}
They arise in various contexts, e.g., number theory, algebraic geometry, algebra, as well as  knot theory. The origins of the modern systematic treatment of MZVs can be traced to \cite{Ecalle81,Zagier94,Hoffman97}. Moreover, MZVs and their generalisations, i.e., multiple polylogarithms, play an important role in quantum field theory \cite{Broadhurst13}.

\smallskip

It is well-known that the nested sums in \eqref{eq:defMZVs} can also be written as iterated Chen integrals, which induces \emph{shuffle relations} among MZVs thanks to integration by parts. Multiplying MZVs directly, on the other hand, yields \emph{quasi-shuffle relations}. Comparison of the different products results in a large collection of so-called \emph{double shuffle relations} among MZVs \cite{Hoffman03,Zudilin03,Hoffman05,Waldschmidt11}. Amending a regularization procedure with respect to the formally defined single zeta value $t:=\zeta(1)$ leads to \emph{regularized double shuffle relations}, which conjecturally give all linear relations among MZVs \cite{Ihara06}.  

A particular change of variables in the context of the integral representation of MZVs gives way to a peculiar set of relations subsumed under the notion of {\it{duality}} for MZVs. As an example we state the identity 
$$
	\zeta(5,1) = \zeta(3,1,1,1).
$$
A precise understanding of the mathematical relation between the notion of duality and the aforementioned double shuffle structure is part of a class of important open problems in the theory of MZVs. See \cite{Ihara06} for details. 
  
\smallskip

Generalizations of the real-valued nested sums in \eqref{eq:defMZVs} to power series in $\bQ[[q]]$ are commonly known as  $q$-analogues of MZVs. A particular example of such $q$-MZVs is due to Bradley and Zhao \cite{Bradley05,Zhao07}, who extended a $q$-analog of the Riemann zeta function introduced by Kaneko et al.~\cite{Kaneko03}. More recently, several $q$-analogues of MZVs were shown to satisfy -- regularized -- double shuffle relations \cite{Castillo13b,Singer14, Singer15,Takeyama13}. The notion of duality in the context of those $q$-MZVs is more complicated. See \cite{Zhao14} for details. Inspired by Bachmann's intriguing work \cite{Bachmann15}, Zudilin presents in \cite{Zudilin15} a particular model called \emph{multiple $q$-zeta brackets}, which possesses a natural quasi-shuffle product. After multiplying Zudilin's multiple $q$-zeta brackets with a certain positive integer power of $1-q$ one obtains ordinary MZVs in the classical limit $q\! \to\!1$. The key result in \cite{Zudilin15} is an algebraic duality-type construction that permits to deduce a shuffle product for multiple $q$-zeta brackets from the quasi-shuffle product of the model. This fact was shown by Bachmann to hold in depth one and two \cite{Bachmann15}. In a nutshell, this construction works as follows: Let $(\cA, m)$ be an algebra, and $\zeta \colon (\cA,m) \to k$ is a multiplicative linear map from $\cA$ into the ring $k$. The map $\tau \colon \cA \to \cA$ is a particular involution, i.e., $\tau \circ \tau= \Id$. Both maps $\zeta$ and $\tau$ are compatible in the sense that $\zeta\circ \tau = \zeta$. Then the \emph{dual product} $m_\square: \cA \otimes  \cA \to \cA$ corresponding to the original product $m$ on $\cA$ is defined by 
\begin{align}
\label{dual-producat}
 m_\square := \tau \circ m \circ (\tau\otimes \tau).
\end{align}
It turns out that $\zeta\colon (\cA,m_\square)\to k$ is again a multiplicative linear map  into $k$. In the case of multiple $q$-zeta brackets Zudilin proved that if $(\cA,m)$ is the quasi-shuffle algebra induced by a sum representation of his model, then the dual product $m_\square$ yields a shuffle product on multiple $q$-zeta brackets. The reason for calling it shuffle product comes from the surprising fact that, in the limit $q \uparrow 1$, the product $m_\square$ reduces to the usual shuffle product of classical MZVs. 

\smallskip

The aim of our work is to explore Zudilin's construction in the context of other $q$-analogues of MZVs as well as classical MZVs. To this end we start by extending \eqref{dual-producat} to general Hopf algebras (Proposition \ref{prop:maindual}). The result of this exercise is then applied in different settings. Firstly, we consider classical MZVs. Note that Zudilin already remarked that in this case the product dual to the quasi-shuffle product does not coincide with the shuffle product for MZVs. However, we show how to relate Zudilin's dual product via the construction of a Hoffman-- Ohno relation to a derivation relation for MZVs (Theorem \ref{theo:derivation}). 

Secondly, we study Zudilin's duality construction in the context of the Schlesinger--Zudilin (SZ) model of $q$-MZVs (see \eqref{eq:defSZ} below) using a notion of duality given by Zhao in \cite{Zhao14}. The specific property of the SZ model is that its quasi-shuffle product is of the same form as that of usual MZVs. Furthermore, we note that a shuffle product for \eqref{eq:defSZ} has already been constructed by the third author in \cite{Singer15} using Rota--Baxter operators, and the resulting double shuffle picture for those $q$-MZVs is well-understood. Our study shows that, analogous to the case of multiple $q$-zeta brackets, the quasi-shuffle product and the shuffle product for Schlesinger--Zudilin $q$-MZVs are related through duality (Theorem \ref{theo:SZdual}). 

Our main application of  Zudilin's duality construction takes place in the context of the Ohno--Okuda--Zudilin (OOZ) $q$-model for MZVs (see \eqref{eq:defOOZ} below). Both the quasi-shuffle-like product as well as the shuffle product were explored by Castillo Medina et al.~in \cite{Castillo13a,Castillo13b}. Note that the former is more involved than that of usual MZVs (see Subsection \ref{ssect:SZ-MZSVs}). We also remark that this model can be defined for any integer argument since the $q$-parameter provides an appropriate regularization \cite{Ebrahimi15}. Whereas for classical MZVs as well as for Schlesinger--Zudilin $q$-MZVs the duality relation is of the form $\zeta \circ \tau = \zeta$, the situation in the case of OOZ $q$-MZVs is different. It turns out that the quasi-shuffle-like product for OOZ $q$-MZVs is related by duality to the shuffle product for SZ $q$-{\it{multiple zeta star values}} ($q$-MZSVs) (see Theorem \ref{theo:OOZdual}). The algebraic structure of the shuffle product in the OOZ $q$-model in turn is induced by the quasi-shuffle product of the SZ $q$-MZSVs. Applying the duality construction leads again to the shuffle product presented in \cite{Castillo13b} (see Theorem \ref{theo:OOZdual} below) as long as we restrict ourselves to non-negative integer arguments. In fact, at this stage we are unable to extend the duality construction relating OOZ $q$-MZVs and SZ $q$-MZSVs to negative or mixed sign integer arguments. 

The OOZ model is also of interest regarding renormalization of MZVs, in a way that is compatible with the shuffle product. Indeed, by applying a theorem due to Connes and Kreimer, we renormalize MZVs in \cite{Ebrahimi15} preserving both the shuffle relations as well as meromorphic continuation. In the present paper we construct an infinitesimal bialgebra, which provides a bialgebra that was originally defined in the context of arbitrary integer arguments. If we restrict this construction to non-negative arguments in the context of the quasi-shuffle product, then the coproduct of convergent words coincides with the coproduct obtained from the duality construction. As a result the left-coideal of convergent words is transferred to a right-coideal (Theorem \ref{theo:infinitesimal}).

\medskip

The paper is organized as follows. In Section \ref{sect:words}, we introduce algebras and subalgebras of noncommutative words over various alphabets which are relevant for classical multiple zeta values and several of their q-analogues. In Section \ref{sect:AutomoHA} we detail the structure maps (multiplication, comultiplication, unit, counit, antipode) of a Hopf algebra, which is obtained by transferring a given Hopf algebra by means of a vector space isomorphism $T$. As a first application we address in Section \ref{sect:MZVs} the case of classical MZVs. Section \ref{sect:qMZVs} is devoted to the study of Hopf algebras transferred by an involution, in the context of particular $q$-analogues of MZVs. Two different involutions $\tau$ and $\widetilde\tau$ are at play, leading to two distinct families of duality relations (Theorem \ref{ooz-duality}), which we express in the Ohno--Okuda--Zudilin model \cite{Ohno12}. The family associated with $\tau$ leads to the classical family of duality relations in the limit $q\to 1$, whereas the second family associated with $\widetilde\tau$, derived from \cite[Theo. 8.3]{Zhao14}, does not seem to have any classical counterpart.

\bigskip

\noindent\textbf{Acknowledgements:}
The first author is supported by a Ram\'on y Cajal research grant from the Spanish government. The second and third author would like to thank the ICMAT for its warm hospitality and gratefully acknowledge support by the Severo Ochoa Excellence Program. The second author is partially supported by Agence Nationale de la Recherche (projet CARMA).\\


\section{Word-algebraic background}
\label{sect:words}

The polynomial algebra in two non-commutative variables $p,y$ is denoted by $\fH:=\bQ\langle p,y \rangle$. The unit of $\bQ\langle p,y \rangle$ is denoted by $\mathbf{1}$. The subalgebra of words not ending in the letter $p$ is denoted by $\fH^1 := \bQ\mathbf{1}\oplus \bQ\langle p,y \rangle y$ and the subalgebra of words not beginning with the letter $y$ is denoted by $\fH^{(-1)} := \bQ\mathbf{1}\oplus p\bQ\langle p,y\rangle$. The subalgebra $\fH^1\cap\fH^{(-1)}= \bQ\mathbf{1} \oplus p \bQ\langle p,y \rangle y$ of words not beginning with the letter $y$ and not ending in $p$ is denoted by $\fH^0$. The subalgebra $\fH^1$ can be expressed as $\bQ\langle z_k\colon k\in \bN_0\rangle$ where $z_k$ stands for the block $p^ky$. The subalgebra $\fH^0$ is linearly spanned by the words  $z_{k_1}\cdots z_{k_n}$ with $k_1 \ge 1$ and  $k_j \in \bN_{0}$ for $j\ge 2$.\\

The polynomial algebra in two non-commutative variables $x_0,x_1$ will be denoted by $\fh:=\bQ\langle x_0,x_1 \rangle$. The unique algebra morphism $\Phi:\fH\to\fh$ such that $\Phi(p):=x_0$ and $\Phi(y):=x_1$ is obviously an isomorphism. We define $\fh^1$, $\fh^{-1}$ and $\fh^0$ similarly as their capitalized counterparts, hence $\fh^1=\Phi(\fH^1)$ and so on. The subalgebra $\fh^1$ can be expressed as $\bQ\langle z_k\colon k\in \bN \rangle$ where $z_k$ stands for the block $x_0^{k-1}x_1$. The subalgebra $\fh^0$ is linearly spanned by the words  $z_{k_1}\cdots z_{k_n}$ with $k_1 \ge 2$ and  $k_j \in \bN$ for $j\ge 2$, called \emph{convergent words} in view of \eqref{eq:defMZVs}. The unique algebra morphism $J:\fh\to \fH$ such that $J(x_0):=p$ and $J(x_1):=py$ is injective. Hence we will consider $\fh$ as a subalgebra of $\fH$ by identifying $x_0$ with $p$ and $x_1$ with $py$. This convention renders both definitions of the letters $z_k,k\in\mathbb N_0$, consistent. Let us remark for later use the inclusions:
\begin{align*}
 \fh\subset\fH^{(-1)},\hskip 12mm \fh^{1}\subset\fH^0,\hskip 12mm \fh^0\subset\fH^0.
\end{align*}

The unique anti-automorphism $\widetilde\tau$ of $\fH$ such that $\widetilde\tau(p):=y$ and $\widetilde\tau(y):=p$ is involutive. It exchanges $\fH^1$ and $\fH^{(-1)}$, and preserves $\fH^0$. Accordingly, the unique anti-automorphism $\tau$ of $\fh$ such that $\tau(x_0):=x_1$ and $\tau(x_1):=x_0$ is involutive, exchanges $\fh^1$ and $\fh^{(-1)}$, and preserves $\fh^0$. We obviously have $\widetilde\tau=\Phi^{-1}\circ\tau\circ\Phi$. The algebra $\fH$ and all subalgebras defined above are graded by the \emph{weight}, defined on the generators by $\wt(p):=1$ and $\wt(y):=0$. Let us remark at this stage that $\tau$ respects the weight whereas $\widetilde\tau$ does not, and that $\widetilde \tau$ does not preserve convergent words, i.e., $\widetilde\tau(\fh^0)\not\subset\fh^0$ .


\section{Transferred Hopf algebra structures}
\label{sect:AutomoHA}

Before starting we briefly recall a few co- and Hopf algebra notions \cite{Kassel95}. Let $C$ be a coalgebra with coproduct $\Delta_C$ and counit $\epsilon_{C}$. A subspace $J \subset C$ is called a left (right) {\it{coideal}} if $\Delta_{C}(J) \subset C \otimes J$ ($\Delta_{C}(J) \subset J \otimes C$). A right (left) {\it{comodule}} over $C$ is a $k$-vector space $M$ together with a linear map $\phi: M \to M \otimes C$, such that $(\Id_{M } \otimes \Delta_{C})\circ \phi = (\phi \otimes \Id_{C}) \circ  \phi$ and $(\Id_{M} \otimes \epsilon_{C})\circ \phi = \Id_{M}$ (analogously for left comodules). Let $(\cH,m,\eta, \Delta, \epsilon,S)$ be a Hopf algebra. The flip map is denoted $s\colon \cH\otimes \cH \to \cH\otimes \cH$, and defined by $s(a\otimes b):=b\otimes a$. Then the \emph{opposite Hopf algebra} $(\cH,m_{op},\eta, \Delta_{op}, \epsilon,S)$ is given by $m_{op}:= m\circ s$ and $\Delta_{op}:=s\circ \Delta$. Further, for the coproduct $\Delta: \cH \to \cH \otimes \cH$ and counit $\epsilon : \cH \to k$ we introduce the notation $\Delta_{\cH\otimes \cH}:=(\Id\otimes\ s \otimes \Id)(\Delta \otimes \Delta)$ and $\epsilon_{\cH\otimes \cH}:=\epsilon \otimes \epsilon$. It will be convenient to apply Sweedler's notation 
\begin{align*}
 \Delta(x)=\sum_{(x)}x_1\otimes x_2. 
\end{align*}

\smallskip 

Let $k$ be a field of characteristic zero and $(\cR,m_{\cR},\eta_{\cR})$ an unital associative $k$-algebra. Our starting point is a Hopf algebra $(\cH ,m,\eta, \Delta,\epsilon,S)$ over $k$ equipped with a $\cR$-valued character, i.e., a multiplicative linear map $\xi \colon \cH \to \cR$, $\xi \circ m  = m_{\cR} \circ (\xi\otimes \xi)$. Subsequently we transfer the whole structure on another $k$-vector space $\cH_\square $ by means of a $k$-linear isomorphism $T \colon \cH_\square \to \cH$. We call $\cH_\square$ the {\it{transferred Hopf algebra ($T$-Hopf algebra)}}. This principle has been used by Zudilin \cite{Zudilin15}.

\begin{proposition}[{{$T$-Hopf algebra}}]{\label{prop:maindual}}
Let $(\cR,m_{\cR},\eta_{\cR})$ be an unital associative $k$-algebra and $(\cH,m,\eta, \Delta,\epsilon,S)$ is a Hopf algebra with character $\xi\colon \cH\to \cR$. We assume the existence of a $k$-linear isomorphism $T \colon \cH_\square \to \cH$ between the $k$-vector spaces $\cH_\square$ and $\cH$. Then $(\cH_\square,m_\square,\eta_\square, \Delta_\square,\epsilon_\square,S_\square)$ is a Hopf algebra with character $\xi_{\square}\colon \cH_\square \to \cR$ with
\begin{itemize}
 \item multiplication: $m_\square\colon \cH_\square \otimes \cH_\square \to \cH_\square$, $m_\square:= T^{-1} \circ m \circ(T\otimes T)$,
 \begin{align*}
 \xymatrix{
 \cH_\square\otimes \cH_\square \ar[d]_{T\otimes T} \ar[r]^-{m_\square} & \cH_\square  \\
 \cH\otimes \cH \ar[r]^-m &   \cH \ar[u]_{T^{-1}}}   
\end{align*}
 \item unit: $\eta_\square\colon k \to \cH_\square$, $\eta_\square:= T^{-1}\circ \eta$,
  \begin{align*}
 \xymatrix{
     & k\ar[dl]_{\eta} \ar[dr]^{\eta_\square} & \\
 \cH \ar[rr]^{T^{-1}} & & \cH_\square}   
\end{align*}
 \item coproduct: $\Delta_\square\colon \cH_\square \to \cH_\square\otimes \cH_\square$, $\Delta_\square:=(T^{-1}\otimes T^{-1})\circ \Delta \circ T$,
  \begin{align*}
 \xymatrix{
 \cH_\square \ar[d]_{T} \ar[r]^-{\Delta_\square} & \cH_\square\otimes \cH_\square  \\
 \cH \ar[r]^-{\Delta} &   \cH \otimes \cH \ar[u]_{T^{-1}\otimes T^{-1}}}   
\end{align*}
 \item counit: $\epsilon_\square \colon \cH_\square \to k$, $\epsilon_\square:=\epsilon \circ T$,
   \begin{align*}
 \xymatrix{
     & \cH_\square\ar[dl]_{T} \ar[dr]^{\epsilon_\square} & \\
 \cH \ar[rr]^{\epsilon} & & k}   
\end{align*}
 \item antipode: $S_\square \colon \cH_\square \to \cH_\square$, $S_\square:=T^{-1}\circ S\circ T$,
  \begin{align*}
 \xymatrix{
 \cH_\square \ar[d]_{T} \ar[r]^{S_\square} & \cH_\square  \\
 \cH \ar[r]^{S} &   \cH \ar[u]_{T^{-1}}}   
\end{align*}
 \item character: $\xi_\square \colon \cH_\square \to \cR$, $\xi_\square:=\xi \circ T$.
  \begin{align*}
 \xymatrix{
     \cH_\square \ar[dr]_{T} \ar[rr]^{\xi_\square} & & \cR \\
 & \cH \ar[ur]_{\xi}  & }   
\end{align*} 
\end{itemize}
\end{proposition}

\noindent The linear isomorphism $T$ is obviously a Hopf algebra isomorphism.

\begin{corollary}\label{cor:maindual}
 In the notations of Proposition \ref{prop:maindual} we have: 
 \begin{enumerate}[a)]
  \item Let $\tilde{\cH} \subset \cH$ be a subalgebra of $(\cH,m,\eta)$, then $T^{-1}(\tilde{\cH}) \subset \cH_\square$ is a subalgebra of $(\cH_\square,m_\square,\eta_\square)$.
  \item Let $\tilde{\cH} \subset \cH$ be a subcoalgebra of $(\cH,\Delta,\epsilon)$, then $T^{-1}(\tilde{\cH}) \subset \cH_\square$ is a subcoalgebra of $(\cH_\square,\Delta_\square,\epsilon_\square)$. 
  \item Let $\tilde{\cH} \subset \cH$ be a (left) coideal of $(\cH,\Delta,\epsilon)$, then $T^{-1}(\tilde{\cH}) \subset \cH_\square$ is a (left) coideal of $(\cH_\square, \Delta_\square,\epsilon_\square)$.
 \end{enumerate}
\end{corollary}


\section{Application of $T$-Hopf algebra to MZVs}
\label{sect:MZVs}

Proposition \ref{prop:maindual} is applied in the context of classical multiple zeta values (MZVs). We briefly review the double-shuffle structure of MZVs and relate the dual product, which is derived from an involution defined on the shuffle Hopf algebra, to a derivation relation of MZVs.


\subsection{Multiple zeta values}
\label{ssect:MZVs}

Thanks to the coexistence of sum and integral representations of MZVs, the $\bQ$-algebra spanned by those values contains many linear relations, called \emph{double shuffle relations}. The latter are most naturally described in terms of the algebraic framework of formal word algebras equipped with shuffle and quasi-shuffle products, which we review briefly.\\

We use the notations of Section \ref{sect:words}. The \emph{quasi-shuffle product} $m_\ast \colon \fh^1 \otimes \fh^1 \to \fh^1$, $u \ast v:=m_*(u \otimes v)$ is iteratively defined as follows:
\begin{enumerate}[(QS1)]
 \item $\mathbf{1} \ast w := w \ast \mathbf{1} := w$ 
 \item $z_nu \ast z_mv:= z_n(u \ast z_mv) + z_m(z_nu \ast v) + z_{n+m}(u\ast v)$
\end{enumerate}
for words $u,v,w\in \fh^1$ and $n,m\in \bN$. Further the \emph{shuffle product} $m_\sh \colon \fh^1 \otimes \fh^1 \to \fh^1$, $u \sh v:=m_\sh(u \otimes v)$, is given by 
\begin{enumerate}[(SH1)]
 \item $\mathbf{1} \sh w := w \sh \mathbf{1}:= w$ 
 \item $a u \sh bv:= a(u \sh bv) + b(au \sh v)$
\end{enumerate}
for words $u,v,w \in \fh^1$ and letters $a,b\in \{x_0,x_1\}$. The difference between both products is seen best by comparing the quasi-shuffle product $z_2 * z_2 = 2 z_2z_2 + z_4$ with the shuffle product $x_0x_1 \sh x_0x_1 = 2 x_0x_1x_0x_1 + 4 x_0x_0x_1x_1.$\\

\noindent Defining the linear map $\z \colon \fh^0 \to \bR$ for any convergent word $x_0^{k_1-1}x_1 \cdots x_0^{k_n-1}x_1 \in \fh^0$ by 
$$
	\z(x_0^{k_1-1}x_1 \cdots x_0^{k_n-1}x_1):=\z(k_1,\ldots,k_n),
$$ 
we have the following well known result. 

\begin{theorem}[\cite{Waldschmidt02,Zudilin03,Hoffman03,Ihara06}]\label{theo:doubleshuffle}
 The map $\z \colon \fh^0\to \bR$ is an algebra morphism with respect to both algebras $(\fh^0,m_\ast)$ and $(\fh^0,m_\sh)$. Moreover, the following so-called  {\rm{extended double shuffle relations}} hold   
 \begin{align}\label{eq:doubleshuffle}
  \z(u \ast v - u \sh v) = 0  \hspace{0.5cm} \text{~and~} \hspace{0.5cm} \z(z_1 \ast w - x_1 \sh w) = 0, 
 \end{align}
for any words $u,v,w\in \fh^0$. 
\end{theorem}
\noindent In fact, conjecturally all linear relations among MZVs are assumed to follow from the relations in \eqref{eq:doubleshuffle}. See \cite{Ihara06}.


\subsection{Duality relations}
\label{ssect:duality}
With the notations of Section \ref{sect:words} at hand, we have the following well-known duality result:

\begin{theorem}[\cite{Hoffman97}]\label{theo:classicalduality}
 For any word $w\in \fh^0$ we have $\z(w)=\z(\tau(w))$. 
\end{theorem}

Finally, we consider the so-called \emph{derivation relations}, which give rise to linear relations among MZVs. For $n\in \bN$ we introduce the derivation $\partial_n \colon \fh^0 \to \fh^0$ by 
\begin{align*}
 \partial_n(x_0) := x_0(x_0+x_1)^{n-1}x_1 \hspace{1cm}\text{and}\hspace{1cm}  \partial_n(x_1):=-x_0(x_0+x_1)^{n-1}x_1. 
\end{align*}
Then we have the next result. 
\begin{theorem}[\cite{Kaneko07}]
  For any word $w\in \fh^0$ and any $n > 0$ we have $\z(\partial_n(w))=0$.
\end{theorem}

Now we apply Proposition \ref{prop:maindual} to the quasi-shuffle Hopf algebra $(\fh^1,m_\ast, \Delta)$, where the coproduct $\Delta\colon \fh^1 \to \fh^1\otimes \fh^1$ is defined by the usual deconcatenation of words
\begin{align}\label{eq:decon}
 \Delta(z_{k_1}\cdots z_{k_n}):= z_{k_1}\cdots z_{k_n} \otimes \mathbf{1} 
 							+ \mathbf{1} \otimes z_{k_1}\cdots z_{k_n} 
 							+ \sum_{l=1}^{n-1} z_{k_1}\cdots z_{k_l} \otimes z_{k_{l+1}}\cdots z_{k_n}.
\end{align}
The product $m_\square: \fh^{(-1)} \otimes \fh^{(-1)} \to \fh^{(-1)}$ is defined in terms of $\tau \colon \fh^{(-1)} \to \fh^1$ and the quasi-shuffle product $m_\ast$ on $\fh^{1}$ 
$$
	m_\square := \tau \circ m_* \circ(\tau\otimes \tau).
$$
Since $\fh^0$ is a subalgebra of $(\fh^1,m_\ast)$, $(\fh^0,m_\square)$ is also a subalgebra of $(\fh^{(-1)},m_\square)$ by Corollary \ref{cor:maindual}. Note that Zudilin remarked in \cite{Zudilin15} that $m_\square$ does not coincide with the usual shuffle product for MZVs defined algebraically through (SH1) and (SH2). This can be seen by noticing that the weight is preserved under both duality as well as the quasi-shuffle product. Depth on the other hand is transformed by duality to the difference of weight and depth. As the quasi-shuffle product leads to a decrease in depth the product $m_\square$ leads to an increase in depth, in contrast to the shuffle product, which preserves both weight and depth. For example
$$
	m_\square(x_0x_1 \otimes x_0x_1 ) = 2 x_0x_1 x_0x_1 + x_0x_1x_1x_1. 
$$

Now we show that the product $m_\square$ is related to a derivation relation for MZVs.

\begin{theorem}\label{theo:derivation}
 Let $w\in\fh^0$. Then we have 
 \begin{align}\label{eq:derrel}
  \partial_2(w) =   w\square z_2 -w\ast z_2.
 \end{align}
\end{theorem}

\begin{remark}
 One should compare Theorem \ref{theo:derivation} with an identity of Hoffman and Ohno proved in \cite{Hoffman03}. They showed that for any $w\in \fh^0$  
 \begin{align*}
  \partial_1(w) = w\sh z_1 - w \ast z_1. 
 \end{align*}
\end{remark}

\begin{proof}[Proof of Theorem \ref{theo:derivation}]
In the first step of the proof we show that  \eqref{eq:derrel} is true for $u=x_0^ax_1^b\in\fh^0$  with $a,b\in \bN$. 
\begin{itemize}
 \item Let $a,b=1$. Then $\partial_2(x_0x_1) = x_0x_1^3-x_0^3x_1$ and 
\begin{align*}
  x_0x_1 \square x_0x_1 -z_2\ast z_2  & =   \tau(2 x_0x_1x_0x_1 + x_0^3x_1 ) - (2 x_0x_1x_0x_1 + x_0^3x_1)  \\
  & = x_0x_1^3-x_0^3x_1. 
\end{align*}
\item Let $b=1$. We obtain by induction hypothesis 
\begin{align*}
 \partial_2(x_0^ax_1) 
 & = \partial_2(x_0)x_0^{a-1}x_1 + x_0\partial_2(x_0^{a-1}x_1)\\
 & = x_0^2x_1x_0^{a-1}x_1 + x_0x_1^2x_0^{a-1}x_1  + x_0( x_0^{a-1}x_1\square x_0x_1 -((x_0^{a-1}x_1)\ast (x_0x_1)).
\end{align*}
Observing 
\begin{align*}
 x_0^ax_1  \square x_0x_1
 & = \tau  ( z_2z_1^{a-1} \ast z_2) \\
 & =  \tau  ( (z_2 z_1^{a-2} \ast z_2)z_1) +  \tau  ( z_2 z_1^{a-1} z_2)+  \tau  ( z_2 z_1^{a-2} z_3)\\
 &  =x_0(x_0^{a-1}x_1 \square x_0x_1) + x_0x_1x_0^ax_1+x_0x_1^2x_0^{a-1}x_1
\end{align*}
and 
\begin{align*}
 x_0^{a}x_1\ast x_0x_1 
 & = (z_{a+1}\ast z_2) = z_2z_{a+1}+z_{a+1}z_2+z_{a+3} \\
 & = x_0(z_az_2+z_2z_a+z_{a+2}) - x_0(z_2z_a) + z_2z_{a+1}\\
 & = x_0(z_a\ast z_2) - x_0^2x_1x_0^{a-1}x_1 + x_0x_1x_0^ax_1,
\end{align*}
concludes the proof.
\item Let $a\in \bN$ be fixed. Then by induction hypothesis
\begin{align*}
 \partial_2(x_0^ax_1^b) 
 & = \partial_2(x_0^ax_1^{b-1})x_1+x_0^{a}x_1^{b-1}\partial_2(x_1) \\
 & =  ( x_0^{a}x_1^{b-1}\square x_0x_1 -x_0^{a}x_1^{b-1}\ast x_0x_1)x_1 - x_0^ax_1^{b-1}x_0^2x_1 - x_0^ax_1^{b-1}x_0x_1^2. 
\end{align*}
We obtain 
\begin{align*}
  x_0^{a}x_1^b\square x_0x_1 
  & = \tau (z_{b+1} z_1^{a-1}\ast z_2) \\
  & = \tau(x_0(z_{b} z_1^{a-1}\ast z_2)) -\tau(x_0(z_2z_{b} z_1^{a-1}))+\tau (z_2z_{b+1}z_1^{a-1}) \\
  & = ( x_0^{a}x_1^{b-1}\square x_0x_1 )x_1  - x_0^ax_1^{b-1}x_0x_1^2+x_0^ax_1^bx_0x_1
\end{align*}
and 
\begin{align*}
 x_0^{a}x_1^{b} \ast x_0x_1
 & = (z_{a+1} z_1^{b-1}\ast z_2) \\
 & = ((z_{a+1}z_1^{b-2}\ast z_2)z_1) +(z_{a+1}z_1^{b-1} z_2)+(z_{a+1}z_1^{b-2} z_3)\\
 & = (x_0^{a}x_1^{b-1}\ast x_0x_1)x_1+x_0^ax_1^bx_0x_1+x_0^ax_1^{b-1}x_0^2x_1,
\end{align*}
which concludes the first part of the proof. 
\end{itemize}

To finish the proof we define the map $\delta(u):= u\square x_0x_1 -u\ast x_0x_1$. For $w_1:=x_0^{a_1}x_1^{b_1}\cdots x_0^{a_n}x_1^{b_n}$ and $w_2:=x_0^{c_1}x_1^{d_1}\cdots x_0^{c_m}x_1^{d_m}$ with $m,n\in \bN$ we will prove that  
\begin{align}\label{eq:derprod}
 \delta(w_1w_2)= \delta(w_1)w_2 + w_1\delta(w_2). 
\end{align}
Indeed, we have 
\begin{align*}
 w_1w_2\square x_0x_1 
  =&~ \tau ( z_{d_m+1}z_1^{c_m-1} \cdots z_{d_1+1}z_1^{c_1-1}z_{b_n+1}z_1^{a_n-1}\cdots z_{b_1+1}z_1^{a_1-1} \ast z_2) \\
   =&~  \tau ( z_{b_n+1}z_1^{a_n-1}\cdots z_{b_1+1}z_1^{a_1-1})\tau ( z_{d_m+1}z_1^{c_m-1} \cdots z_{d_1+1}z_1^{c_1-1} \ast z_2) \\
   &+\tau ( z_{b_n+1}z_1^{a_n-1}\cdots z_{b_1+1} z_1^{a_1-1} \ast z_2) \tau(z_{d_m+1}z_1^{c_m-1} \cdots z_{d_1+1}z_1^{c_1-1}) \\
   &-\tau( z_{d_m+1}z_1^{c_m-1} \cdots z_{d_1+1}z_1^{c_1-1}z_2 z_{b_n+1}z_1^{a_n-1}\cdots z_{b_1+1}z_1^{a_1-1})\\
   =&~ w_1 (w_2\square x_0x_1 ) + (w_1\square x_0x_1)w_2 -w_1 x_0x_1 w_2
\end{align*}
as well as 
\begin{align*}
 w_1w_2\ast x_0x_1
  =& ~  z_{a_1+1}z_1^{b_1-1}\cdots z_{a_n+1}z_1^{b_n-1} z_{c_1+1}z_1^{d_1-1}\cdots z_{c_m+1}z_1^{d_m-1}\ast z_2\\
  = & ( z_{a_1+1}z_1^{b_1-1}\cdots z_{a_n+1}z_1^{b_n-1})(z_{c_1+1}z_1^{d_1-1} \cdots z_{c_m+1}z_1^{d_m-1}\ast z_2)\\
   &+ ( z_{a_1+1}z_1^{b_1-1}\cdots z_{a_n+1}z_1^{b_n-1}\ast z_2)(z_{c_1+1}z_1^{d_1-1} \cdots z_{c_m+1}z_1^{d_m-1})\\
   & - ( z_{a_1+1}z_1^{b_1-1}\cdots z_{a_n+1}z_1^{b_n-1} z_2z_{c_1+1}z_1^{d_1-1} \cdots z_{c_m+1}z_1^{d_m-1})\\
   =&~ w_1 (w_2 \ast x_0x_1) +  (w_1\ast x_0x_1)w_2 - w_1x_0x_1 w_2. 
\end{align*}
This implies \eqref{eq:derprod}. Combining the first and the second step proves the claim of the theorem. 
\end{proof}


\section{Application of $T$-Hopf algebra to the shuffle product of $q$-MZVs}
\label{sect:qMZVs}

A systematic study of $q$-analogues of MZVs ($q$-MZVs) was initiated by Bradley, Zhao and Zudilin \cite{Bradley05,Zhao07,Zudilin03}, with a forerunner by Schlesinger's 2001 preprint \cite{Schlesinger01}. Since then several distinct $q$-models of MZVs have been considered in the literature. See Zhao's recent work \cite{Zhao14} for details. The quasi-shuffle products of these $q$-MZVs are deduced from the defining series, and in several cases an accompanying Hopf algebra structure in the sense of Hoffman \cite{Hoffman12} can be defined. Quite recently $q$-shuffle products for the most prevailing models of $q$-MZVs were studied in  \cite{Castillo13a,Castillo13b,Singer15,Singer14}. See also \cite{Zhao14}. They are based on a description of $q$-MZVs in terms of iterated Rota--Baxter operators. The latter enter the picture through substituting Jackson's $q$-integral for the Riemann integral in the integral representation of ordinary MZVs.  

In this section we compare the Rota--Baxter operator approach and the duality approach for the $q$-analogues of  Schlesinger--Zudilin (SZ-) model and Ohno--Okuda--Zudilin (OOZ-) model. In the forthcoming we always assume that $q \in \bC$ with $|q|<1$. The notation $q\uparrow 1$ stands for $q\to 1$ inside some angular sector $-\pi/2+\varepsilon<\hbox{Arg}(1-q)<\pi/2-\varepsilon$ centered at $1$, with a fixed $\varepsilon>0$.


\subsection{Schlesinger--Zudilin $q$-model}
\label{ssect:SZ-qMZVs}

In \cite{Schlesinger01} and \cite{Zudilin03} the authors defined the following $q$-analogues of MZVs given by
\begin{align}\label{eq:defSZ}
 \z^{\operatorname{SZ}}(k_1,\ldots,k_n)&:=\sum_{m_1>\cdots >m_n>0} \frac{q^{m_1k_1+\cdots +m_nk_n}}{(1-q^{m_1})^{k_1}\cdots (1-q^{m_n})^{k_n}}, 
\end{align} 
 for $k_1,\ldots,k_n \in \bN_0$ with $k_1\geq 1$. It is easily seen that for $k_1\geq 2$ and $k_2,\ldots,k_n\geq 1$ we obtain 
 \begin{align*}
\lim_{q \uparrow  1} (1-q)^{k_1+\cdots+k_n}\z^{\operatorname{SZ}}(k_1,\ldots,k_n) = \z(k_1,\ldots,k_n). 
 \end{align*}
 Furthermore we introduce a star version of model \eqref{eq:defSZ} which is given for any $k_1,\ldots,k_n\in \bN_0$ with $k_1\geq 1$ by 
 \begin{align}\label{eq:defSZstar}
\z^{{\operatorname{SZ}}, \star}(k_1,\ldots,k_n)&:=\sum_{m_1\geq\cdots \geq m_n>0} \frac{q^{m_1k_1+\cdots +m_nk_n}}{(1-q^{m_1})^{k_1}\cdots (1-q^{m_n})^{k_n}}.
\end{align}
This model plays a distinguished role with respect to the OOZ-model, which will be discussed in the next paragraph. 
Again we observe for $k_1\geq 2$ and $k_2,\ldots,k_n\geq 1$ that
 \begin{align*}
\lim_{q \uparrow  1} (1-q)^{k_1+\cdots+k_n}\z^{{\operatorname{SZ}},\star}(k_1,\ldots,k_n) = \z^\star(k_1,\ldots,k_n):=\sum_{m_1 \ge \cdots \ge m_n>0}\frac{1}{m_1^{k_1}\cdots m_n^{k_n}}, 
 \end{align*}
 where $\z^\star(k_1,\ldots,k_n)$ are called \emph{multiple zeta star values} (MZSVs).

\begin{remark}
Note that \eqref{eq:defSZ} constitutes what is usually called a modified (Schlesinger--Zudilin) $q$-MZV, and the proper $q$-MZV is given by $(1-q)^{k_1+\cdots+k_n}\z^{\operatorname{SZ}}(k_1,\ldots,k_n)$. The same applies to the $q$-MZSV defined in \eqref{eq:defSZstar}. In the following we will work exclusively with \eqref{eq:defSZ}  and \eqref{eq:defSZstar}.    
\end{remark}

\smallskip

Using the notations of Section \ref{sect:words} again, we define the $\lambda$-weighted quasi-shuffle type product $m^{(\lambda)}_{\ast} \colon \fH^1\otimes \fH^1 \to \fH^1$, $u \ast_{\lambda} v:=m^{(\lambda)}_{\ast} (u \otimes v)$, iteratively by 
\begin{enumerate}[(QS1)]
 \item $\be \ast_{\lambda} w :=w \ast_{\lambda} \be:= w$, 
 \item $z_n u \ast_{\lambda} z_m v := z_n (u \ast_{\lambda} z_m v) + z_m(z_n u \ast_{\lambda} v)+\lambda z_{n+m}(u\ast_{\lambda} v)$
\end{enumerate}
for any words $w,u,v\in \fH^1$ and $n,m\in \bN_0$. The parameter $\lambda$ will stand for $1$ or $-1$. The case $\lambda=1$ corresponds to the natural product satisfied by the $q$-MZVs defined in \eqref{eq:defSZ}. It coincides with the quasi-shuffle product of classical MZVs. Note that $(\fH^1,m^{(1)}_\ast,\Delta)$, where $\Delta$ is given by deconcatenation \eqref{eq:decon}, defines a quasi-shuffle Hopf algebra. The case $\lambda=-1$ corresponds to the product satisfied by the $q$-MZSVs defined in \eqref{eq:defSZstar}. Now we introduce the $\lambda$-weighted shuffle product $m^{(\lambda)}_{\sh} \colon \fH^1 \otimes \fH^1 \to \fH^1$, $u \sh_{\lambda} v:=m^{(\lambda)}_{\sh} (u \otimes v)$, by 
\begin{enumerate}[(SH1)]
 \item $\be\sh_\lambda w := w\sh_\lambda \be:= w$, 
 \item $yu \sh_\lambda v:= u\sh_\lambda yv:= y(u\sh v)$,
 \item $pu \sh_\lambda pv:=p(u \sh_\lambda pv) + p(pu \sh_\lambda v) + \lambda p(u\sh_\lambda v)$.
\end{enumerate}
\smallskip
\noindent Defining the map $\zeta^{\operatorname{SZ}}\colon \fH^0\to \bQ[[q]]$ by 
$$
	\zeta^{\operatorname{SZ}}(z_{k_1}\cdots z_{k_n}):=\zeta^{\operatorname{SZ}}(k_1,\ldots,k_n)
$$ 
we can state the following result of \cite{Singer15}: 

\begin{theorem}
The map $\z^{\operatorname{SZ}}\colon \fH^0\to \bQ[[q]]$ is an algebra morphism on both algebras $(\fH^0,m^{(1)}_{\ast})$ and $(\fH^0,m^{(1)}_{\sh})$. Especially, for any words $u,v\in \fH^0$ we have the $q$-analogue of the double shuffle relation
\begin{align*}
 \z^{\operatorname{SZ}}(u\sh_{1} v - u\ast_{1} v) = 0. 
\end{align*} 
\end{theorem}

First we introduce the linear map $\widetilde\tau \colon \bQ\langle p,y \rangle \to \bQ\langle p,y \rangle$ 
\begin{align}\label{eq:antiq}
	\widetilde\tau(p):=y \qquad \widetilde\tau(y):=p,
\end{align}
which is extended to words as an antiautomorphism. Again, this induces the linear isomorphism $\widetilde\tau \colon \fH^{(-1)} \to \fH^{1}$. By further restriction we obtain the automorphism $\widetilde\tau\colon \fH^0 \to \fH^0$, which yields the following duality result due to Zhao.

\begin{theorem}\cite[Theo. 8.3]{Zhao14}\label{thm:zhao}
 For any word $w\in \fH^0$ we have $\z^{\operatorname{SZ}}(w)=\z^{\operatorname{SZ}}\big(\widetilde\tau(w)\big)$.
\end{theorem}

\noindent The first nontrivial relation arising from Theorem \ref{thm:zhao} is:
\begin{equation*}
\zeta^{\operatorname{SZ}}(2)=\zeta^{\operatorname{SZ}}(ppy)=\zeta^{\operatorname{SZ}}(pyy)=\zeta^{\operatorname{SZ}}(1,0)=\sum_{k,l>0}(k-1)q^{kl}.
\end{equation*}

 \begin{proof}
We include brief demonstration of Theorem \ref{thm:zhao}. As in \cite{Singer15} we introduce the Rota--Baxter operator 
 \begin{align*}
  \oP_q[f](t):=\sum_{m \ge 1}f(q^mt)
 \end{align*}
 of weight one. Further let $y(t):=\frac{t}{1-t}\in t\bQ[[t]]$. Using the identity \cite[Prop. 2.6 (i)]{Singer15}
 \begin{align*}
  \z^{\operatorname{SZ}}(p^{k_1}y\cdots p^{k_n}y) =\left. \oP_q^{k_1}[y \oP_q^{k_2}[y \cdots \oP_q^{k_n}[y]\cdots]](t)\right|_{t=1}  
 \end{align*}
 we calculate 
 \begin{align*}
  \lefteqn{\z^{\operatorname{SZ}}(\tau_q(p^{k_1}y\cdots p^{k_n}y)) = \z^{\operatorname{SZ}}(py^{k_n}\cdots py^{k_1})} \\
   &=~ \left. \oP_q[y^{k_n} \oP_q[y^{k_{n-1}} \cdots \oP_q[y^{k_1}]]](t)\right|_{t=1}\\
   &=~\left. \sum_{m_1,\ldots,m_n> 0} \frac{q^{(m_1+\cdots +m_n)k_1}t^{k_1}}{(1-q^{m_1+\cdots+m_n}t)^{k_1}} \frac{q^{(m_2+\cdots +m_n)k_2}t^{k_2}}{(1-q^{m_2+\cdots+m_n}t)^{k_2}}\cdots \frac{q^{m_nk_n}t^{k_n}}{(1-q^{m_n}t)^{k_n}}\right|_{t=1} \\
   &=~ \sum_{m_1> m_2 > \cdots > m_n >0} \frac{q^{m_1k_1+\cdots +m_nk_n}}{(1-q^{m_1})^{k_1}\cdots(1-q^{m_n})^{k_n} } \\
   &=~ \z^{\operatorname{SZ}}(p^{k_1}y\cdots p^{k_n}y). 
 \end{align*}
 \end{proof}

Again, we are in the position to apply Proposition \ref{prop:maindual} in the context of the quasi-shuffle Hopf algebra $(\fH^1,m^{(1)}_\ast,\Delta)$ and the character $\z^{\operatorname{SZ}}$. Since $\fH^0$ is a subalgebra of $\fH^1$ it is also a subalgebra of $\fH_\square^1$, and we obtain the algebra $(\fH^0, m^{(1)}_\square)$. Here we find the same situation as in \cite{Zudilin15}, i.e., the product  $m^{(1)}_\square = \widetilde\tau \circ m^{(1)}_\ast \circ (\widetilde\tau \otimes \widetilde\tau)$, $a \square_{1} b:= m^{(1)}_\square (a \otimes b)$ equals the shuffle product $m^{(1)}_\sh$.

\begin{theorem}\label{theo:SZdual}
 The product $m^{(1)}_\square \colon \fH^0\otimes \fH^0\to \fH^0$ induced by Proposition \ref{prop:maindual} coincides with $m^{(1)}_\sh\colon \fH^0\otimes \fH^0\to \fH^0$. In particular, we have the double shuffle relations
 \begin{align*}
 \z^{\operatorname{SZ}}(u\square_{1} v - u\ast_{1} v) = 0
\end{align*} 
for any words $u,v\in \fH^0$.
\end{theorem}

\begin{proof}
Let $X:=\{p,y\}$. For $u\in \fH^0$ we observe that $ \be \square_1 u = u \square_1 \be = \widetilde\tau (\be \ast_1 \widetilde\tau(u)) = u.$ Let $u',v'\in pX^\ast y$ then there exist $a,b,c,d\geq 1$ and $u,v\in pX^\ast y\cup \{\be\}$ such that $u'=p^ay^b u$ and $v'=p^cy^d v$. We observe that the product
 \begin{align*}
  p^ay^b u \square_{1} p^cy^d v  = \widetilde\tau \left(\widetilde\tau (p^ay^b u) \ast_1 \widetilde\tau(p^cy^d v)\right) 
   =  \widetilde\tau \left(\widetilde\tau(u) z_b z_0^{a-1} \ast_1 \widetilde\tau(v)z_dz_0^{c-1}\right).
 \end{align*}
Now we distinguish three cases: 
 \begin{itemize}
  \item Case 1: $a=c=1$ 
   \begin{align*}
  py^b u \square_1 py^d v  =&~  \widetilde\tau \left(\widetilde\tau(u) z_b \ast_1 \widetilde\tau(v)z_d\right)\\
   = &~  \widetilde\tau \left\{(\widetilde\tau(u) \ast_1 \widetilde\tau(v)z_d)z_b 
   				+ (\widetilde\tau(u) z_b \ast_1 \widetilde\tau(v))z_d\right. 
				+  \left.(\widetilde\tau(u) \ast_1 \widetilde\tau(v))z_{b+d}\right\}\\
   = &~  py^b (u\square_1 py^d v) + py^d (py^b u \square_1 v)  +py^{b+d}(u\square_1 v).
 \end{align*}
 \item Case 2: $a=1,c\geq 2$ 
    \begin{align*}
  py^b u \square_{1} p^cy^d v  =
  &~  \widetilde\tau  \left(\widetilde\tau(u) z_b \ast_1 \widetilde\tau(v)z_dz_0^{c-1}\right)\\
  = &~  \widetilde\tau  \left\{(\widetilde\tau(u) \ast_1 \widetilde\tau(v)z_dz_0^{c-1})z_b
  	+ (\widetilde\tau(u) z_b \ast_1 \widetilde\tau(v)z_dz_0^{c-2})z_0\right.  
	+  \left.(\widetilde\tau(u) \ast_1 \widetilde\tau(v)z_dz_0^{c-2})z_{b}\right\} \\
  = &~  py^b (u\square_1 p^cy^d v) + p (py^b u \square_1 p^{c-1}y^dv) + py^b(u\square_1 p^{c-1}y^dv).
 \end{align*}
 The case $a\geq 2,c=1$ is analogues. 
 \item Case 3: $a,c\geq 2$ 
     \begin{align*}
  \lefteqn{p^ay^bu \square_1 p^cy^dv  =
   \widetilde\tau  \left(\widetilde\tau(u) z_bz_0^{a-1} \ast_1 \widetilde\tau(v)z_dz_0^{c-1}\right)}\\
   = &~  \widetilde\tau  \left\{(\widetilde\tau(u)z_bz_0^{a-2} \ast_1 \widetilde\tau(v)z_dz_0^{c-1})z_0  
   		+ (\widetilde\tau(u) z_bz_0^{a-1} \ast_1 \widetilde\tau(v)z_dz_0^{c-2})z_0 \right.
   		+  \left.(\widetilde\tau(u)z_b z_0^{a-2} \ast_1 \widetilde\tau(v)z_dz_0^{c-2})z_0\right\}\\
   = &~  p (p^{a-1}y^bu\square_1 p^cy^d v) 
   		+ p (p^ay^b u \square_1 p^{c-1}y^dv) 
   		+ p(p^{a-1}y^b u\square_1 p^{c-1}y^dv).
 \end{align*}
 \end{itemize}
It is easy to see that the three cases reduce to (SH1), (SH2) and (SH3), respectively.   
\end{proof}


\subsection{Ohno--Okuda--Zudilin $q$-model}
\label{ssect:OOZ-qMZVs}

Now we consider another $q$-model introduced by Ohno, Okuda and Zudilin (OOZ) in \cite{Ohno12}. Let $k_1,\ldots,k_n\in \bZ$. Then the (modified) OOZ $q$-MZVs are defined by
\begin{align}\label{eq:defOOZ}
 \z^{\operatorname{OOZ}}(k_1,\ldots,k_n):=\sum_{m_1>\cdots >m_n>0} \frac{q^{m_1}}{(1-q^{m_1})^{k_1}\cdots (1-q^{m_n})^{k_n}}.
\end{align}
Again, this is a proper $q$-analogue of MZVs thanks to the fact that the limit
\begin{align*}
 \lim_{q\uparrow 1}(1-q)^{k_1+\cdots+k_n}\z^{\operatorname{OOZ}}(k_1,\ldots,k_n) = \z(k_1,\ldots,k_n), 
\end{align*}
if $k_1\geq 2$, $k_2,\ldots,k_n\geq 1$. 

If we restrict ourselves to non-negative integers then we can introduce an algebraic setting analogues to the one for the Schlesinger--Zudilin model presented in the previous section. 
We define 
\begin{align*}
 \z^{\operatorname{OOZ}}\colon \fH^0\to \bQ[[q]], 
 	\hspace{0.5cm} 
	\z^{\operatorname{OOZ}}(z_{k_1}\cdots z_{k_n}):=\z^{\operatorname{OOZ}}(k_1,\ldots,k_n) 
\end{align*}
and 
\begin{align*}
 \z^{{\operatorname{SZ}},\star}\colon \fH^0\to \bQ[[q]], 
 	\hspace{0.5cm} 
	\z^{{\operatorname{SZ}},\star}(z_{k_1}\cdots z_{k_n}):=\z^{{\operatorname{SZ}},\star}(k_1,\ldots,k_n). 
\end{align*}

Using the linear map $\widetilde\tau$ defined in \eqref{eq:antiq} we can state the following duality theorem. 

\begin{theorem}\label{theo:dualOOZ}
For any word $w \in \fH^0$ we find that $\z^{\operatorname{OOZ}}(w) = \z^{\operatorname{SZ},\star} (\widetilde\tau(w))$.
\end{theorem}

\begin{proof}
We use the Rota--Baxter operator $P_q = \Id + \oP_q$, which explicitly writes 
\begin{align}\label{eq:RotaBaxterP}
  P_q[f](t):=\sum_{m \ge 0}f(q^mt)
\end{align}
of weight $-1$, together with $y(t):=\frac{t}{1-t}\in t\bQ[[t]]$. Following \cite[Eq. (4.1)]{Castillo13b} we can write
 \begin{align*}
  \z^{\operatorname{OOZ}}(p^{k_1}y\cdots p^{k_n}y) =\left.P_q^{k_1}[yP_q^{k_2}[y \cdots P_q^{k_n}[y]\cdots]](t)\right|_{t=q},  
 \end{align*}
and calculate 
 \begin{align*}
  \lefteqn{\z^{\operatorname{OOZ}}(\widetilde\tau(p^{k_1}y\cdots p^{k_n}y)) = \z^{\operatorname{OOZ}}(py^{k_n}\cdots py^{k_1})} \\
   &=~ \left. P_q[y^{k_n}P_q[y^{k_{n-1}} \cdots P_q[y^{k_1}]]](t)\right|_{t=q}\\
   &=~\left. \sum_{m_1,\ldots,m_n\geq 0} \frac{q^{(m_1+\cdots +m_n)k_1}t^{k_1}}{(1-q^{m_1+\cdots+m_n}t)^{k_1}} \frac{q^{(m_2+\cdots +m_n)k_2}t^{k_2}}{(1-q^{m_2+\cdots+m_n}t)^{k_2}}\cdots \frac{q^{m_nk_n}t^{k_n}}{(1-q^{m_n}t)^{k_n}}\right|_{t=q} \\
   &=~ \sum_{m_1\geq m_2 \geq \cdots \geq m_n >0} \frac{q^{m_1k_1+\cdots +m_nk_n}}{(1-q^{m_1})^{k_1}\cdots(1-q^{m_n})^{k_n} } \\
   &=~ \z^{SZ,\star}(p^{k_1}y\cdots p^{k_n}y). 
 \end{align*}
 Since $\widetilde\tau$ is an antiautomorphism for words from $\bQ\langle p,y \rangle$, the proof is complete. 
 \end{proof}
 
In general $\widetilde\tau$ does not preserve the weight of a given word. Therefore we can not deduce the classical duality relation of MZVs stated in Theorem \ref{theo:classicalduality} by taking the limit $q\uparrow 1$ in the previous theorem. 
Subsequently, we prove that in leading $q$-order a $q$-version of Theorem \ref{theo:classicalduality} for the OOZ-model is valid (Corollary \ref{coro:OOZdualc}).
Recall that for $k_1\geq 2$, $k_2,\ldots,k_n \geq 1$ the (modified) Bradley--Zhao (BZ) $q$-MZVs are defined by the iterated sum
\begin{align*}
 \z^{{\operatorname{BZ}}}(k_1,\ldots,k_n):=\sum_{m_1 > \cdots > m_n>0}\frac{q^{(k_1-1)m_1+\cdots + (k_n-1)m_n}}{(1-q^{m_1})^{k_1}\cdots (1-q^{m_n})^{k_n}}. 
\end{align*}

For this model we have the following well-known result: 
\begin{theorem}\cite{Bradley05}\label{theo:BradleyDualWord}
 For any word $w\in \fh^0$ we have $\z^{\operatorname{BZ}}(w)=\z^{\operatorname{BZ}}\big(\tau(w)\big)$.
\end{theorem}
\noindent The notation $\z^{\operatorname{BZ}}(z_{k_1}\cdots z_{k_n})=\z^{\operatorname{BZ}}(x_0x_1^{k_1-1}\cdots x_0x_1^{k_n-1}):=\z^{\operatorname{BZ}}(k_1,\ldots,k_n)$ is used. It will be useful to reformulate both Theorems  \ref{thm:zhao} and \ref{theo:BradleyDualWord} in the OOZ model. Let us introduce the linear maps $U:\fh^0\to\fh^0$ and $V:\fH^0\to\fH^0$ defined by
\begin{equation}\label{triangular}
U(z_{k_1}\cdots z_{k_n}):=\sum_{{2\le r_1\le k_1 \atop 1\le r_j\le k_j,\,j\ge 2}}\binom{k_1-2}{r_1-2}\binom{k_2-1}{r_2-1}\cdots\binom{k_n-1}{r_n-1}z_{r_1}\cdots z_{r_n}
\end{equation}
and
\begin{equation}\label{triangular-bis}
V(z_{k_1}\cdots z_{k_n}):=\sum_{{1\le r_1\le k_1 \atop 0\le r_j\le k_j,\,j\ge 2}}\binom{k_1-1}{r_1-1}\binom{k_2}{r_2}\cdots\binom{k_n}{r_n}z_{r_1}\cdots z_{r_n}.
\end{equation}
\begin{proposition}\label{tranfer-bz-ooz}
We have the following:
\begin{enumerate}
\item The maps $U$ and $V$ are linear isomorphisms, with inverses given by
\begin{equation}\label{triangular-inv}
U^{-1}(z_{k_1}\cdots z_{k_n}):=\sum_{{2\le r_1\le k_1 \atop 1\le r_j\le k_j,\,j\ge 2}}(-1)^{\sum_j k_j-r_j}\binom{k_1-2}{r_1-2}\binom{k_2-1}{r_2-1}\cdots\binom{k_n-1}{r_n-1}z_{r_1}\cdots z_{r_n}
\end{equation}
and
\begin{equation}\label{triangular-inv-bis}
V^{-1}(z_{k_1}\cdots z_{k_n}):=\sum_{{1\le r_1\le k_1 \atop 0\le r_j\le k_j,\,j\ge 2}}(-1)^{\sum_j k_j-r_j}\binom{k_1-1}{r_1-1}\binom{k_2}{r_2}\cdots\binom{k_n}{r_n}z_{r_1}\cdots z_{r_n}.
\end{equation}
\item For convergent words, the Bradley--Zhao and Ohno--Okuda--Zudilin models are related by
\begin{equation}
\z^{\operatorname{OOZ}}=\z^{\operatorname{BZ}}\circ U.
\end{equation}
\item For words in $\fH^0$, the Schlesinger--Zudilin and Ohno--Okuda--Zudilin models are related by
\begin{equation}
\z^{\operatorname{OOZ}}=\z^{\operatorname{SZ}}\circ V.
\end{equation}

\end{enumerate}
\end{proposition}
\begin{proof}
Let $U'$, respectively $V'$, be the linear endomorphism of $\fh^0$, respectively $\fH^0$, given by the right-hand side of \eqref{triangular-inv}, respectively \eqref{triangular-inv-bis}. We compute
\begin{eqnarray*}
{\lefteqn{(U'\circ U)(z_{k_1}\cdots z_{k_n})=\sum_{{2\le r_1\le k_1 \atop 1\le r_j\le k_j,\,j\ge 2}}\binom{k_1-2}{r_1-2}\binom{k_2-1}{r_2-1}\cdots\binom{k_n-1}{r_n-1}U'(z_{r_1}\cdots z_{r_n})}}\\
&=&\sum_{{2\le s_1\le r_1\le k_1 \atop 1\le s_j\le r_j\le k_j,\,j\ge 2}}(-1)^{\sum_j r_j-s_j}\binom{k_1-2}{r_1-2}\binom{k_2-1}{r_2-1}\cdots\binom{k_n-1}{r_n-1}\binom{r_1-2}{s_1-2}\binom{r_2-1}{s_2-1}\cdots\binom{r_n-1}{s_n-1}z_{s_1}\cdots z_{s_n}\\
&=&\sum_{{2\le s_1\le k_1 \atop 1\le s_j \le k_j,\,j\ge 2}}
D_{k_1-2}^{s_1-2}D_{k_2-1}^{s_2-1}\cdots D_{k_n-1}^{s_n-1}z_{s_1}\cdots z_{s_n},
\end{eqnarray*}
with $D_k^a:=\sum_{a+b+c=k}(-1)^{b}\frac{k!}{a!b!c!}$. It is easily seen by expanding $x^k=(x-y+y)^k$ that $D_k^a$ reduces to the Kronecker symbol $\delta_k^a$. Hence $(U'\circ U)(z_{k_1}\cdots z_{k_n})=z_{k_1}\cdots z_{k_n}$. The verifications of $U\circ U'=\hbox{Id}_{\fh^0}$, $V'\circ V=\hbox{Id}_{\fH^0}$ and $V\circ V'=\hbox{Id}_{\fH^0}$ are entirely similar. The second assertion easily follows from the two following identities:
\begin{align*}
  \frac{q^{m}}{(1-q^m)^k}  = \sum_{l=2}^k\binom{k-2}{l-2}\frac{q^{(l-1)m}}{(1-q^m)^l} \quad \text{and} \quad
  \frac{1}{(1-q^m)^k}  =  \sum_{l=1}^k\binom{k-1}{l-1}\frac{q^{(l-1)m}}{(1-q^m)^l}.
\end{align*}
The third assertion follows from:
\begin{align*}
  \frac{q^{m}}{(1-q^m)^k}  = \sum_{l=1}^k\binom{k-1}{l-1}\frac{q^{lm}}{(1-q^m)^l} \quad \text{and} \quad
  \frac{1}{(1-q^m)^k}  =  \sum_{l=0}^k\binom{k}{l}\frac{q^{lm}}{(1-q^m)^l}.
\end{align*}
\end{proof}
\begin{remark}
We obviously have
\begin{equation*}
\z^{\operatorname{BZ}}=\z^{\operatorname{OOZ}}\circ U^{-1},\hskip 12mm \z^{\operatorname{SZ}}=\z^{\operatorname{OOZ}}\circ V^{-1}.
\end{equation*}
This could have been directly checked through the following identities:
\begin{align*}
  \frac{q^{(k-1)m}}{(1-q^m)^k} & = \sum_{l=2}^k(-1)^{k-l}\binom{k-2}{l-2}\frac{q^m}{(1-q^m)^l}, \quad\quad\quad \frac{q^{(k-1)m}}{(1-q^m)^k}  =  \sum_{l=1}^k(-1)^{k-l}\binom{k-1}{l-1}\frac{1}{(1-q^m)^l},&\\
   \frac{q^{km}}{(1-q^m)^k} & = \sum_{l=1}^k(-1)^{k-l}\binom{k-1}{l-1}\frac{q^m}{(1-q^m)^l}, \quad\quad\quad
  \frac{q^{km}}{(1-q^m)^k}  =  \sum_{l=0}^k(-1)^{k-l}\binom{k}{l}\frac{1}{(1-q^m)^l}.&
\end{align*}
\end{remark}
\noindent Hence we obtain two families of duality relations in the OOZ model:
\begin{theorem}\label{ooz-duality}
Let $w\in\fH^0$. Then
\begin{equation*}
\z^{\operatorname{OOZ}}(w)=\z^{\operatorname{OOZ}}\big(V^{-1}\circ\widetilde\tau\circ V(w)\big).
\end{equation*}
If moreover $w\in\fh^0$ we also have
\begin{equation*}
\z^{\operatorname{OOZ}}(w)=\z^{\operatorname{OOZ}}\big(U^{-1}\circ\tau\circ U(w)\big).
\end{equation*}
\end{theorem}
\noindent For example, from $V^{-1}\circ\widetilde\tau\circ V(z_3)=z_1z_0^2+2z_1z_0+z_1$ and $U^{-1}\circ\tau\circ U(z_3)=z_2z_1+z_2$ we get:
\begin{eqnarray*}
\z^{\operatorname{OOZ}}(3)&=&\z^{\operatorname{OOZ}}(1)+2\z^{\operatorname{OOZ}}(1,0)+\z^{\operatorname{OOZ}}(1,0,0)\\
&=& \z^{\operatorname{OOZ}}(2,1)+\z^{\operatorname{OOZ}}(2)\\
&=&\sum_{k,l>0}\frac{k(k+1)}{2}q^{kl}.
\end{eqnarray*}
\noindent The second equality can also be obtained as a double shuffle relation \cite{Castillo13b}. Theorems \ref{thm:zhao} and \ref{theo:BradleyDualWord} can be reformulated as follows:
\begin{corollary}\label{theo:BradleyDual}
For $a_1,b_1,\ldots,a_n,b_n\in \bN$ we have 
\begin{align*}
\z^{{\operatorname{SZ}}}(a_1,\{0\}^{b_1-1},\ldots, a_n,\{0\}^{b_n-1}) 
 	&= \z^{{\operatorname{SZ}}}(b_n,\{0\}^{a_n-1},\ldots, b_1,\{0\}^{a_1-1}),\\
 \z^{{\operatorname{BZ}}}(a_1+1,\{1\}^{b_1-1},\ldots, a_n+1,\{1\}^{b_n-1}) 
 	&= \z^{{\operatorname{BZ}}}(b_n+1,\{1\}^{a_n-1},\ldots, b_1+1,\{1\}^{a_1-1}). 
\end{align*}
\end{corollary}

\noindent Using this reformulation and Theorem \ref{ooz-duality} we immediately get:

\begin{corollary}\label{coro:OOZdualc}
We have 
 \begin{align*}
 \z^{\operatorname{OOZ}}(a_1+1,\{1\}^{b_1-1}, \ldots,a_n+1,\{1\}^{b_n-1})&= \z^{\operatorname{OOZ}}(b_n+1,\{1\}^{a_n-1}, \ldots,b_1+1,\{1\}^{a_1-1})+A(q), \\
\z^{\operatorname{OOZ}}(a_1,\{0\}^{b_1-1}, \ldots,a_n,\{0\}^{b_n-1})&= \z^{\operatorname{OOZ}}(b_n,\{0\}^{a_n-1}, \ldots,b_1,\{0\}^{a_1-1})+B(q),
 \end{align*}
 where $A(q)\in \langle \z^{\operatorname{OOZ}}(k_1,\ldots,k_n)\colon k_1\geq 2, k_2,\ldots,k_n\geq 1 \rangle_\bQ$ and $B(q) \in \langle \z^{\operatorname{OOZ}}(k_1,\ldots,k_n)\colon k_1\geq 1, k_2,\ldots,k_n\geq 0 \rangle_\bQ$. Moreover we have
 \begin{align*}
  \lim_{q \uparrow 1} (1-q)^{a_1+b_1+\cdots + a_n+b_n} A(q)=0. 
 \end{align*}
\end{corollary}

\begin{proof}
This is immediate. The last assertion on the term $A(q)$ is derived from the triangular structure of $U$, i.e., $U(w)=w+\hbox{terms of strictly lower weight}$. The term $B(q)$ is not so nice, due to the fact that the duality $\widetilde \tau$ does not preserve the weight.
\end{proof}

Let us go back to the duality construction of Theorem \ref{theo:dualOOZ}. For the OOZ-model we apply Proposition \ref{prop:maindual} in the context of the Hopf algebra  $(\fH^1,m^{(-1)}_\ast, \Delta)$ with character $\z^{{\operatorname{SZ}},\star}$, where $\Delta$ is defined as in \eqref{eq:decon}. We obtain the dual Hopf algebra $(\fH^{(-1)},m^{(-1)}_\square,\Delta_{\square})$. Again, $\fH^0$ is a subalgebra of $(\fH^{(-1)},m^{(-1)}_*)$ and we have the following theorem. 

\begin{theorem}\label{theo:OOZdual}
 The product $m^{(-1)}_\square\colon \fH^0\otimes \fH^0 \to \fH^0$ induced by Proposition \ref{prop:maindual} coincides with the shuffle product $m^{(-1)}_\sh\colon \fH^0\otimes \fH^0 \to \fH^0$. 
\end{theorem}

\begin{proof}
The proof is very similar to that of Theorem \ref{theo:SZdual}. We only have to consider the weight $-1$ quasi-shuffle and shuffle products instead of the weight $1$ products.
\end{proof}

\begin{remark}
Note that the $q$-MZVs in the OOZ-model \eqref{eq:defOOZ} are well defined for any integer arguments. However, we are not able to extend the duality argument of Theorem \ref{theo:OOZdual} to arbitrary integers, and have to restrict to non-negative arguments. The shuffle product approach provided in \cite{Castillo13b} is valid for any integer. The reason for this is that the Rota--Baxter operator $P_q$ defined in \eqref{eq:RotaBaxterP} is invertible.  This fact was used in \cite{Ebrahimi15} to provide a renormalization of MZVs with respect to the shuffle product, which is compatible with meromorphic continuation of MZVs.  
\end{remark}

\subsection{Comparison of different $q$-models}
\label{ssect:comparison}

In this paragraph we discuss the connection of the shuffle products for the SZ- and OOZ-models. The duality construction unveils that the link corresponds to the interplay between multiple zeta and multiple zeta star values. 

Following \cite{Ihara11} we define the map $S\colon \fH^1\to \fH^1$ recursively by $ S(\mathbf{1})=\mathbf{1}$ and
\begin{align*}
\hspace{1cm} S(z_k w)= z_k S(w) + z_k \circ S(w)
\end{align*}
for any $k\in \bN$ and $w\in \fH^1$. The composition $\circ\colon \fH^1 \otimes \fH^1 \to \fH^1$ is defined by $z_k\circ \be = 0$ and
\begin{align*}
 z_{k_1}\circ (z_{k_2}w) = z_{k_1+k_2}w \hspace{0.5cm} 
\end{align*}
for any $w \in \fH^1$ and $k_1,k_2\in \bN_0$. 

Then we have the following result: 

\begin{theorem} \cite{Ihara11} \label{theo:Ihara}
We have: 
\begin{itemize}
 \item The map $S$ is an isomorphism with inverse $S^{-1}$ given by $  S^{-1}(\be) = \be$
 \begin{align*}
	S^{-1}(z_kw) = z_kS^{-1}(w)-z_k\circ S^{-1}(w)
 \end{align*}
 for any $k\in \bN_0$ and $w\in \fH^1$.
 \item The map $S\colon (\fH^1,m^{(-1)}_\ast) \to (\fH^{1},m^{(1)}_\ast)$ is an algebra isomorphism, i.e., 
 \begin{align*}
  S(u\ast_{-1}v)=S(u) \ast_{1} S(v)
 \end{align*}
for any $u,v\in \fH^1$. 
\end{itemize} 
\end{theorem}

Therefore, putting these results together, we obtain: 

\begin{corollary}
 The following diagram commutes: 
 \begin{table}[h]
 \begin{align*}
 \xymatrix{
 \fH^0\otimes \fH^0 \ar[d]_{\widetilde\tau\otimes \widetilde\tau} \ar[rr]^{\sh_{-1}}& & \fH^0  \\
  \fH^0\otimes \fH^0 \ar[rr]^{\ast_{-1}} \ar[d]_{S\otimes S}& & \fH^0\ar[d]^{S} \ar[u]_{\widetilde\tau} \\ 
  \fH^0\otimes \fH^0 \ar[rr]^{\ast_1} & & \fH^0\ar[d]^{\widetilde\tau} \\
  \fH^0\otimes \fH^0 \ar[u]^{\widetilde\tau \otimes \widetilde\tau} \ar[rr]^{\sh_{1}}& &  \fH^0}
\end{align*}
\end{table}
\end{corollary}

\begin{proof}
 From top to bottom the commutativity of the diagram is ensured by Theorems \ref{theo:OOZdual}, \ref{theo:Ihara} and \ref{theo:SZdual}.  
\end{proof}


\section{Further applications of the $T$-Hopf algebra to $q$-MZVs}

\subsection{A shuffle-like bialgebra structure for the OOZ-model}
\label{ssect:shuffle-bialgebra}

In \cite{Castillo13a,Castillo13b} the authors provided a description of the shuffle product for the OOZ-model. The main ingredient is a Rota--Baxter operator (RBO) approach, which substitutes for the Kontsevich integral formula in the classical MZVs case. Since the corresponding RBO is invertible, the shuffle product can be extended to arbitrary integer arguments.  

First we quickly review the algebra framework. Let $X:=\{p,d,y\}$, and let $W$ denote the set of words on the alphabet $X$ subject to the rule $pd=dp=\be$, where $\be$ denotes the empty word. Let $W_0$ be the subset of $W$ made of words ending in the letter $y$. Therefore any word $w\in W_0$ can be written as 
\begin{align*}
 w=p^{k_1}yp^{k_2}y\cdots p^{k_n}y
\end{align*}
for $k_1,\ldots,k_n\in \bZ$ using the identifications $p^{-1}=d$ and $p^0=\be$. Let $\cH$ (respectively $\cH_0$) denote the algebra $\cH:=\langle W \rangle_\bQ$  (respectively  $\cH_0:=\langle W_0 \rangle_\bQ$) spanned by the words in $W$ (respectively  $W_0$). For any $\lambda\in\bQ^{\times}$ we define the bilinear product $\sh_{\lambda}$ on $\cH$ by $\be\sh_\lambda w=w\sh_\lambda\be=w$ for any $w\in W$, and recursively for any words $u,v\in W$:
\begin{align*}
yu\sh_{\lambda}v:=u\sh_{\lambda}yv&:=y(u\sh_{\lambda}v),\\
pu\sh_{\lambda}pv&:=p(pu\sh_{\lambda}v+u\sh_{\lambda}pv+\lambda u\sh_{\lambda}v),\\
du\sh_{\lambda}dv&:=\frac{1}{\lambda}\left[d(u\sh_{\lambda}v)-u\sh_{\lambda}dv-du\sh_{\lambda}v\right],\\
du\sh_{\lambda}pv:=pv\sh_{\lambda}du&:=d(u\sh_{\lambda}pv)-u\sh_{\lambda}v-\lambda du\sh_{\lambda}v.
\end{align*}

\begin{lemma}\cite{Castillo13b} \label{lem:shalgebra}
The pair $(\cH,m^{(\lambda)}_\sh)$ is a commutative, associative and unital algebra, and $(\cH_0,m^{(\lambda)}_\sh)$ is a subalgebra of $(\cH,m^{(\lambda)}_\sh)$.
\end{lemma}

Next we define a  \textsl{unital infinitesimal bialgebra} on $\cH$. This will yield a coproduct $\oDelta$, which does not depend on $\lambda$. Then we will show that $(\cH, m^{(\lambda)}_\sh,\oDelta)$ is a proper bialgebra.

Recall that a unital infinitesimal bialgebra is a triple $(\cA,m_\cdot,\oDelta)$ where $(\cA,m_\cdot)$ is a unital associative algebra, and $\oDelta : \cA \to \cA \otimes \cA$ is coproduct with the following compatibility relation \cite{Loday06}:
%
%
\begin{equation}\label{eq:inf}
\oDelta(u \cdot v)=(u\otimes \be)\oDelta(v)+\oDelta(u)(\be\otimes v)-u\otimes v.
\end{equation}
Our ansatz is as follows:
\begin{align}
 \oDelta(\be)	&:=\be\otimes\be,\label{eq:ansatz1}\\
 \oDelta(p)		&:=p\otimes\be+\be\otimes p,\label{eq:ansatzp}\\
 \oDelta(y)		&:=y\otimes\be,\label{eq:ansatzy} \\
 \oDelta(d)		&:=0.\label{eq:ansatzd} 
\end{align}

This is motivated by the idea that $\oDelta$ is supposed to extend in some sense the deconcatenation coproduct to $W$. Here the letters $p$ and $py$ correspond to the letters $x_0$ and $x_1$, respectively, which are primitives with respect to the ususal deconcatenation coproduct.  

\begin{proposition}
The triple $(\cH, m^{(\lambda)}_\sh,\oDelta)$ is a bialgebra, where $\oDelta\colon  \cH \to \cH\otimes \cH$ is defined through \eqref{eq:inf} together with the initial values \eqref{eq:ansatz1}, \eqref{eq:ansatzp}, \eqref{eq:ansatzy} and \eqref{eq:ansatzd}. 
\end{proposition}

\begin{proof}
First we have to prove that $\oDelta$ is well-defined. The initial values \eqref{eq:ansatz1}, \eqref{eq:ansatzp} and \eqref{eq:ansatzd} are consistent with the convention $dp =pd=\be$ since
\begin{align*}
 \oDelta(dp) & = (d\otimes \be)(\be \otimes p + p \otimes \be) -d\otimes p = \be \otimes \be, \\
 \oDelta(pd) & = (\be \otimes p + p \otimes \be)(\be\otimes d) - p\otimes d  = \be \otimes \be.
\end{align*}
Next we show that for a given word $w$ relation \eqref{eq:inf} is independent of the splitting point, i.e., for $w=w_1w_2 = w_1'w_2'$ we have 
\begin{align*}
 \oDelta(w_1w_2) = \oDelta(w_1'w_2').
\end{align*}
This is done by induction on the weight $\wt(w)$. The base case is trivial. Now assume that $w = w_1w_2w_3$ with $\wt(w_1)$, $\wt(w_2)$, $\wt (w_3)$, $\wt(w_2w_3)$,  and $\wt(w_1w_2)$ all being smaller than $\wt(w)$. Splitting $w$ between $w_1$ and $w_2$ we observe 
\begin{align*}
 \oDelta(w_1\cdot w_2w_3) =& ~(w_1\otimes \be)\oDelta(w_2w_3) + \oDelta(w_1)(\be\otimes w_2w_3)- w_1\otimes w_2w_3 \\
 =&~ (w_1w_2\otimes \be) \oDelta(w_3) + (w_1\otimes \be)\oDelta(w_2)(\be \otimes w_3) - w_1w_2\otimes w_3\\
 & + \oDelta(w_1)(\be\otimes w_2w_3)- w_1\otimes w_2w_3 
\end{align*}
and splitting $w$ between $w_2$ and $w_3$ leads to  
\begin{align*}
  \oDelta(w_1w_2\cdot w_3) =& ~(w_1w_2\otimes \be)\oDelta(w_3) + \oDelta(w_1w_2)(\be\otimes w_3)- w_1w_2\otimes w_3 \\
 =&~ (w_1w_2\otimes \be) \oDelta(w_3) + (w_1\otimes \be)\oDelta(w_2)(\be \otimes w_3) + \oDelta(w_1)(\be\otimes w_2w_3) \\
 &  - w_1\otimes w_2 w_3    - w_1w_2\otimes w_3, 
\end{align*}
which coincide by induction hypothesis. 

In the next step we prove coassociativity of $\oDelta$ by induction on the weight $\wt(w)$. The base case is trivial. Now we have to distinguish three cases: 
\begin{itemize}
 \item Case 1: $w=yu$
 \begin{align*}
  ((\oDelta\otimes \Id) \circ \oDelta)(yu) 
  & = (\oDelta\otimes \Id)(y\otimes \be)\oDelta(u)  = \sum_{(u)}\oDelta(yu_1) \otimes u_2\\
  & = (y\otimes \be\otimes \be) (\oDelta\otimes \Id)\oDelta(u) = (y\otimes \be\otimes \be) (\Id \otimes \oDelta)\oDelta(u) \\
  & =  (\Id \otimes \oDelta) (y\otimes \be)\oDelta(u)=((\oDelta\otimes \Id) \circ \oDelta)(yu) .
 \end{align*}
  \item Case 2: $w=pu$\\
  We have
 \begin{align*}
  ((\oDelta\otimes \Id) \circ \oDelta)(pu) &= (\oDelta\otimes \Id) \left[(p\otimes \be)\oDelta(u)+\be\otimes pu\right] \\
  & = \sum_{(u)}\oDelta(pu_1)\otimes u_2 + \be\otimes \be\otimes pu \\
  & = \sum_{(u)} \left[(p\otimes \be)\oDelta(u_1) +\be \otimes pu_1 \right]\otimes u_2+ \be\otimes \be\otimes pu \\
  & = (p\otimes \be \otimes \be) (\oDelta \otimes \Id) \oDelta(u) + \be \otimes ((p \otimes \be)\oDelta(u))+ \be\otimes \be\otimes pu 
 \end{align*}
 and
 \begin{align*}
  ((\Id \otimes \oDelta) \circ \oDelta)(pu) &= (\Id\otimes \oDelta) \left[(p\otimes \be)\oDelta(u)+\be\otimes pu\right] \\
  & = \sum_{(u)}pu_1\otimes \oDelta(u_2) + \be\otimes \oDelta(pu) \\
  & = (p\otimes \be \otimes \be) (\Id \otimes \oDelta) \oDelta(u) + \be \otimes ((p\otimes \be)\oDelta(u))+\be\otimes \be \otimes pu,
 \end{align*} 
 which coincide by induction hypothesis. 
 \item Case 3: $w=du$
 \begin{align*}
  ((\oDelta\otimes \Id) \circ \oDelta)(du) & = (\oDelta\otimes \Id) \left[(d\otimes \be)\oDelta(u)-d\otimes u  \right] \\
  & = \sum_{(u)}\oDelta(du_1)\otimes u_2  = \sum_{(u)}\left[(d\otimes \be) \oDelta(u_1)\otimes u_2  -d\otimes u_1\otimes u_2 \right] \\
  & = (d\otimes \be \otimes \be) (\oDelta \otimes  \Id)\oDelta(u)- d\otimes \oDelta(u) \\
  & =(d\otimes \be \otimes \be) (\Id \otimes \oDelta)\oDelta(u)- d\otimes \oDelta(u) \\
  & = (\Id \otimes \oDelta) \left[(d\otimes \be)\oDelta(u)-d\otimes u  \right] \\
  & = ((\Id \otimes \oDelta) \circ \oDelta)(du).
 \end{align*}
\end{itemize}

We show $\oDelta(u)\sh_\lambda\oDelta(v)=\oDelta(u\sh_\lambda v)$ by induction on the sum of weights $s=\hbox{wt}(u)+\hbox{wt}(v)$, the case $s=1$ being trivial.
\begin{itemize}
\item Case 1: $u=yu'$. Then we have:
\begin{align*}
\oDelta(yu'\sh_\lambda v)
& = \oDelta\big(y(u'\sh_\lambda v)\big)\\
& = \oDelta(y) \big(\be\otimes (u'\sh_\lambda v)\big)+(y\otimes\be)\oDelta(u'\sh_\lambda v)-y\otimes(u'\sh_\lambda v)\\
& = (y\otimes\be)\oDelta(u'\sh_\lambda v)\\
& = (y\otimes\be)(\oDelta(u')\sh_\lambda \oDelta(v))\\
& = ((y\otimes\be) \oDelta(u'))\sh_\lambda \oDelta(v)\\
& = \oDelta(yu')\sh_\lambda\oDelta (v).
\end{align*}
\item Case 2: $u=du'$ and $v=dv'$. Two subcases occur:
 We compute:
\begin{align*}
\oDelta(du'\sh_{\lambda}dv')
&=\frac{1}{\lambda}\oDelta\big(d(u'\sh_{\lambda}v')-du'\sh_{\lambda}v'-u'\sh_{\lambda}dv'\big)\\
&=\frac{1}{\lambda}\Big[\oDelta(d)\big((\be\otimes(u'\sh_{\lambda}v')\big)+(d\otimes\be)\oDelta(u'\sh_{\lambda}v')-d\otimes(u'\sh_{\lambda}v')\\
   &\hskip 8mm -\oDelta(du')\sh_{\lambda}\oDelta(v')-\oDelta(u')\sh_\lambda \oDelta (dv')\Big]\\
   &=\frac{1}{\lambda}\Big[(d\otimes \be)(\oDelta(u')\sh_{\lambda}\oDelta(v'))-d\otimes(u'\sh_{\lambda}v')\\
   &-\big((d\otimes\be)\oDelta(u')-d\otimes u'\big)\sh_{\lambda}\oDelta(v')-\oDelta(u')\sh_\lambda \big((d\otimes\be)\oDelta(v')-d\otimes v'\big)\Big]\\
   &=\frac{1}{\lambda}\Big[(d\otimes \be)\oDelta(u')\sh_{\lambda}\oDelta(v')+\oDelta(u')\sh_\lambda(d\otimes\be)\oDelta(v')+\lambda(d\otimes \be)\oDelta(u')\sh_\lambda (d\otimes\be)\oDelta(v')\\
  &\hskip 8mm -d\otimes(u'\sh_{\lambda}v')\\
  &\hskip 8mm-\big((d\otimes\be)\oDelta(u')-d\otimes u'\big)\sh_{\lambda}\oDelta(v')-\oDelta(u')\sh_\lambda \big((d\otimes\be)\oDelta(v')-d\otimes v'\big)\Big]\\
  &=\frac{1}{\lambda}\big[(d\otimes u')\shl \oDelta(v')+\oDelta(u')\shl (d\otimes v')-d\otimes (u'\shl v')\big]\\
  &\hskip 8mm +(d\otimes \be)\oDelta(u')\sh_\lambda (d\otimes\be)\oDelta(v').
\end{align*}
On the other hand,
\begin{align*}
\oDelta(du')\shl \oDelta(dv')&=\big((d\otimes \be)\oDelta(u')-d\otimes u')\big)\shl \big((d\otimes \be)\oDelta(v')-d\otimes v')\big)\\
&=(d\otimes \be)\oDelta(u')\shl (d\otimes \be)\oDelta(v')-(d\otimes \be)\oDelta(u')\shl (d\otimes v')\\
&\hskip 8mm -(d\otimes u')\shl (d\otimes \be)\oDelta(v') +(d\otimes u')\shl (d\otimes v').
\end{align*}
In view of $d\shl d=-\frac 1\lambda d$ and, more generally, $d\shl du=-\frac 1\lambda d\shl u$ for any $u$, we get:
\begin{align*}
&\oDelta(du'\sh_{\lambda}dv')-\oDelta(du')\shl \oDelta(dv')\\
&=(d\otimes u')\shl \big(\frac 1\lambda -d\otimes\be)\oDelta(v')\big)+\big(\frac 1\lambda -d\otimes\be)\oDelta(u')\big)\shl (d\otimes v')\\
&= \sum_{(v')}(\frac 1\lambda d\shl v'_1-d\shl dv'_1)\otimes (u'\shl v'_2)
	+\sum_{(u')}(\frac 1\lambda d\shl u'_1-d\shl du'_1)\otimes (u'_2\shl v')\\
&= 0.
\end{align*}
\item Case 3: $u=pu'$ and $v=pv'$. We observe
\begin{align*}
 \oDelta(pu')\shl \oDelta(pv') & =  \oDelta(p(u'\shl pv' + p u'\shl v' + \lambda u'\shl v')) \\
 & = (p\otimes \be)\left[ \oDelta(u')\shl \oDelta(pv') + \oDelta(pu') \shl \oDelta(v')+\lambda \oDelta(u')\shl \oDelta(v') \right] \\
 & + (\be\otimes p) (\be \otimes (u'\shl pv' + p u'\shl v' + \lambda u'\shl v'))\\
 & = (p\otimes \be)\left[ \oDelta(u')\shl (p\otimes \be)\oDelta(v')+ \oDelta(u')\shl (\be\otimes pv')\right. \\
 &+ \left.((p\otimes \be) \oDelta(u')) \shl \oDelta(v') + (\be \otimes pu') \shl \oDelta(v')+\lambda \oDelta(u')\shl \oDelta(v') \right] \\
 & + \be \otimes (pu'\shl pv')
\end{align*}
and 
\begin{align*}
 \oDelta(pu') \shl \oDelta(pv') & = ((p\otimes \be)\oDelta(u')+\be \otimes pu')\shl ((p\otimes \be)\Delta(v') + \be\otimes pv')\\
 & = (p\otimes \be) \oDelta(u') \shl (p\otimes \be)\oDelta(v') + \be \otimes (pu'\shl pv')\\
 & + (p\otimes \be) \left[ \oDelta(u')\shl(\be \otimes pv') + (\be \otimes pu')\shl \oDelta(v') \right]\\
 & = (p\otimes \be)\left[ \oDelta(u') \shl (p\otimes \be)\oDelta(v')+ (p\otimes \be) \oDelta(u') \shl \oDelta(v') +\lambda \oDelta(u') \shl \oDelta(v')\right] \\
 & + (p\otimes \be) \left[ \oDelta(u')\shl(\be \otimes pv') + (\be \otimes pu')\shl \oDelta(v') \right] + \be \otimes (pu'\shl pv')
\end{align*}
\item Case 4: $u=du'$ and $v=pv'$. We obtain
\begin{align*}
 \oDelta(du'\shl pv') & = \oDelta(d(u'\shl pv')-u'\shl v' - \lambda du'\shl v') \\
 & = (d\otimes \be) (\oDelta(u')\shl \oDelta(pv')) - \oDelta(u')\shl \oDelta(v') - \lambda \oDelta(du')\shl \oDelta(v') - d\otimes (u'\shl pv')\\
 & = (d\otimes \be) (\oDelta(u')\shl (p\otimes \be)\oDelta(v')) + (d\otimes \be) (\oDelta(u')\shl (\be \otimes pv')) \\
 & - \oDelta(u')\shl \oDelta(v') - \lambda \oDelta(du')\shl \oDelta(v') - d\otimes (u'\shl pv')\\
  & = (d\otimes \be) (\oDelta(u')\shl (p\otimes \be)\oDelta(v')) + (d\otimes \be) (\oDelta(u')\shl (\be \otimes pv'))  - \oDelta(u')\shl \oDelta(v') \\
  & - \lambda (d\otimes \be)\oDelta(u')\shl \oDelta(v')+ \lambda (d\otimes u')\shl \oDelta(v')- d\otimes (u'\shl pv')
\end{align*}
and 
\begin{align*}
 \oDelta(du')\shl\oDelta(pv') & = \left[(d\otimes \be)\oDelta(u')-d\otimes u'\right]\shl \left[(p\otimes \be)\oDelta(v')+\be \otimes pv'\right] \\
 & = (d\otimes \be)\oDelta(u')\shl (p\otimes \be)\oDelta(v') - (d\otimes \be)\oDelta(u') \shl (\be \otimes pv')\\
 & - (d\otimes u')\shl (p\otimes \be)\oDelta(v') - d\otimes (u'\shl pv')\\
 & = (d\otimes \be) (\oDelta(u')\shl(p\otimes \be)\oDelta(v')) - \oDelta(u')\shl\oDelta(v') - \lambda (d\otimes \be)\oDelta(u')\shl \oDelta(v') \\
 & - (d\otimes \be) (\oDelta(u')\shl (\be\otimes pv')) - (d\otimes u')\shl (p\otimes \be)\oDelta(v') - d\otimes (u'\shl pv').
\end{align*}
Using $(d\otimes u')\shl(p\otimes \be)\oDelta(v')=-\lambda (d\otimes u')\shl \oDelta(v')$ we conclude the proof. 
\end{itemize}
\end{proof}

Let $(\fH^1,m_\ast^{(-1)},\Delta)$ be the quasi-shuffle Hopf algebra of weight $-1$. Then $\fH^0$ is a left coideal of $(\fH^1,\Delta)$ and Corollary \ref{cor:maindual} implies that $\fH^0$ is a left coideal of $(\fH^{(-1)},\Delta_\square)$.  

\begin{theorem}\label{theo:infinitesimal}
Let $(\fH^1,m_\ast^{(-1)},\Delta)$ be the quasi-shuffle Hopf algebra of weight $-1$. The coproduct map $\Delta_{\square,op}$ restricted to the left coideal $\fH^0$ arising from the duality construction with respect to $\widetilde\tau$, coincides with the coproduct $\oDelta$ restricted to $\fH^0$ defined by the infinitesimal bialgebra construction, i.e., the following diagram commutes:  
  \begin{align*}
 \xymatrix{
 \fH^0 \ar[dd]_{\widetilde\tau} \ar[r]^-{\oDelta} & \fH^{(-1)} \otimes \fH^0 \\
  &   \fH^0 \otimes \fH^{(-1)} \ar[u]_s \\
  \fH^0  \ar[r]^-{\Delta} & \fH^0 \otimes \fH^1 \ar[u]_{\widetilde\tau \otimes \widetilde\tau}}   
\end{align*}  
\end{theorem}

\begin{proof}
 We prove the theorem by the length of the words. We have 
 \begin{align*}
  \Delta_{\square,op}(p) & = s(\widetilde\tau\otimes \widetilde\tau) \Delta \widetilde\tau(p) = s(\widetilde\tau\otimes \widetilde\tau) \Delta (y) \\
  & = s(\widetilde\tau\otimes \widetilde\tau) (\be \otimes y + y\otimes \be) = \be \otimes p + p\otimes \be\\
  & = \oDelta(p).
 \end{align*}
In the inductive step we distinguish two cases. 
\begin{itemize}
 \item First case: $w'=pw$ where $w$ starts with a $p$. We observe
\begin{align*} 
 \oDelta(pw)  
 & = (p\otimes \be) \oDelta(w) + \oDelta(p)(\be\otimes w)-p\otimes w \\
 & = (p\otimes \be) \Delta_{\square,op}(w) + \be \otimes pw \\
 & = (p\otimes \be) s(\widetilde\tau\otimes \widetilde\tau) \Delta(\widetilde\tau(w)) + \be \otimes pw \\
 & =  s(\widetilde\tau\otimes \widetilde\tau)  \left[\Delta(\widetilde\tau(w))(\be \otimes y) + \widetilde\tau(pw) \otimes \be  \right]\\
 & =  s(\widetilde\tau\otimes \widetilde\tau) \Delta(\widetilde\tau(pw))\\
 & = \Delta_{\square,op}(pw).
\end{align*}
\item Second case: $w'=p\tilde{w}$ where $\tilde{w}=y^{k}w$ ($k\in \bN$) and $w$ starts with $p$ or is the empty word.
\begin{align*}
 \oDelta(py^kw) 
 & = (py^k\otimes \be) \oDelta(w) + \oDelta(py^k)(\be\otimes w)-py^k\otimes w \\
 & = (py^k\otimes \be) \Delta_{\square,op}(w) + \be \otimes py^kw \\ 
 & = (py^k\otimes \be) s(\widetilde\tau\otimes \widetilde\tau) \Delta(\widetilde\tau(w)) + \be \otimes py^kw \\
 & =  s(\widetilde\tau\otimes \widetilde\tau)  \left[\Delta(\widetilde\tau(w))(\be \otimes p^ky) + \widetilde\tau(py^kw) \otimes \be  \right]\\
 & =  s(\widetilde\tau\otimes \widetilde\tau) \Delta(\widetilde\tau(py^kw))\\
 & = \Delta_{\square,op}(py^kw),
\end{align*}
using the fact that $\oDelta(py^k) = \be \otimes py^k + py^k \otimes \be$ ($k\in \bN$) which is easily verified by induction. 
\end{itemize}
\end{proof}


\subsection{The Schlesinger--Zudilin $q$-multiple zeta star vales}
\label{ssect:SZ-MZSVs}

As we have shown in Theorem \ref{theo:OOZdual}, the quasi-shuffle Hopf algebra of weight $-1$, which corresponds to the quasi-shuffle product of Schlesinger--Zudilin $q$-MZSVs, induces a shuffle product for the OOZ-model. In this section we show that we can do the reverse, i.e., we can use the quasi-shuffle-like product of the OOZ-model to define a shuffle product for Schlesinger--Zudilin $q$-MZSVs. 

In the first step, following \cite{Castillo13b}, we introduce the quasi-shuffle product of the OOZ-model. We define the operator $T\colon \fH^0 \to \fH^1$ by 
\begin{align*}
 T(z_mw):=z_mw - z_{m-1}w
\end{align*}
for any $w\in \fH^0$ and $m\in \bN$. Then we define the product $m_\ast^{\operatorname{OOZ}}\colon \fH^0\otimes \fH^0 \to \fH^0$, $m_\ast^{\operatorname{OOZ}}(u\otimes v)=:u\ast_{\operatorname{OOZ}} v$, by
\begin{enumerate}[(QS1)]
 \item $\be \ast_{\operatorname{OOZ}} w := w \ast_{\operatorname{OOZ}} \be := w$; 
 \item $z_m u \ast_{\operatorname{OOZ}} z_n v:= z_m(u\ast_1 T(z_nv)) + z_n(T(z_mu)\ast_1 v) + (z_{m+n}-z_{m+n-1})(u\ast_1v)$
\end{enumerate}
for any $w\in \fH^0$, $u,v\in \fH^1$ and $m,n\in \bN$, where $\ast_1$ was defined in Subsection \ref{ssect:SZ-qMZVs}.

\begin{proposition}[\cite{Castillo13b}]
 The map $\z^{\operatorname{OOZ}}\colon (\fH^0,m_\ast^{\operatorname{OOZ}})\to \bQ[[q]]$ is a character. 
\end{proposition}

Following \cite{Muneta09} we review the shuffle product of classical multiple zeta star values. The shuffle product $\sh_\star \colon \bQ\langle x_0,x_1 \rangle \otimes \bQ\langle x_0,x_1 \rangle\to \bQ\langle x_0,x_1 \rangle $ is defined iteratively by 
\begin{enumerate}[(i)]
 \item $\be \sh_\star u = u\sh_\star \be = u$ 
 \item $a u \sh_\star bv = a (u \sh_\star bv) + b(a u \sh_\star v) - \delta(u)\tau(a)bv - \delta(v)\tau(b)au$
\end{enumerate}
for any words $u,v\in \{x_0,x_1\}^\ast$ and letters $a,b\in \{x_0,x_1\}$, where $\delta$ is defined by 
\begin{align*}
 \delta(w):=\begin{cases}
             1 & w=\mathbf{1}, \\
             0 & \text{otherwise.}
            \end{cases}
\end{align*}

We apply Proposition \ref{prop:maindual} to the algebra $(\fH^0,m_\ast^{\operatorname{OOZ}})$ because for the coalgebra structure we have to go over to the completion of $\fH^0$ since the map $T$ is not bijective. 

We have the following result: 
\begin{theorem}\label{theo:SZSdual}
The product $m_\square^{\operatorname{OOZ}}\colon \fH^0\otimes \fH^0 \to \fH^0$ induced by Proposition \ref{prop:maindual} defines a shuffle product for the Schlesinger--Zudilin $q$-MZSVs, which coincides with $\sh_\star$ modulo terms of lower weight.  
\end{theorem}

\begin{proof}[Proof of Theorem \ref{theo:SZSdual}]
We perform two steps. First we need an alternative description of the shuffle product $\shuffle_\star$ in terms of $\shuffle$:

\begin{lemma}\label{lem:altstarsh}
 The shuffle product $\sh_\star \colon \bQ\langle x_0,x_1 \rangle \otimes \bQ\langle x_0,x_1 \rangle\to \bQ\langle x_0,x_1 \rangle $ can be calculated in terms of the ordinary shuffle product $\sh$ via the formula
 \begin{align*}
   ua \sh_\star vb = ua \sh vb - (u \sh v \tau(b))a - (u\tau(a) \sh v)b.
 \end{align*}
\end{lemma}

\begin{proof}
 We define $ua \tsh vb := ua \sh vb - (u \sh v \tau(b))a - (u\tau(a) \sh v)b$ and prove that $ua \tsh vb = ua \sh_\star vb$ by induction over the sum of the weights of $ua$ and $vb$. For the base case we observe 
 \begin{align*}
  y\ssh y & = 2y^2 - 2xy = y\tsh y,\\
  y\ssh x & = yx+xy - x^2 -y^2 = y\tsh x, \\
  x\ssh x & = 2x^2 -2 yx = x\tsh x.  
 \end{align*}
For the inductive step we have to consider two cases: 
In the first case we obtain
\begin{align*}
 a_1ua_2 \ssh b = 
  & ~ a_1(ua_2 \ssh b) + b a_1ua_2 - \tau(b)a_1ua_2 \\
 =& ~ a_1(ua_2\sh b) -a_1(u\sh \tau(b))a_2 - a_1u \tau(a_2)b  +ba_1ua_2 -\tau(b)a_1ua_2 \\
 =& ~ a_1ua_2\sh b - \left[ a_1(u\sh \tau(b))+ \tau(b) a_1 u  \right]a_2 - a_1u\tau(a_2)b \\
 =& ~ a_1ua_2\sh b - (a_1u\tsh \tau(b))a_2 - a_1 u \tau(a_2)b. 
\end{align*}
Secondly, we observe that
\begin{align*}
 a_1 u a_2 \ssh b_1 u b_2 =& ~ a_1( u a_2 \ssh b_1 u b_2) + b_1 (a_1 u a_2 \ssh u b_2)  \\ & -\delta(ua_2)\tau(a_1) b_1 v b_2 - \delta(vb_2) \tau(v_1)a_1 u a_2 \\
 =& ~ a_1( u a_2 \ssh b_1 u b_2) + b_1 (a_1 u a_2 \ssh u b_2) \\
 =& ~ a_1(ua_2 \sh b_1v_1b_2) - a_1(u \sh b_1v\tau(b_2))a_2 - a_1(u\tau(a_2)\sh b_1v)b_2 \\
  & + b_1(a_1ua_2 \sh vb_2) - b_1(a_1 u \sh v\tau(b_2))a_2 - b_1(a_1u\tau(a_2)\sh v)b_2 \\
 =& ~ a_1ua_2 \sh b_1 v_1 b_2 - \left[ a_1(u\tau(a_2)\sh b_1v) +b_1(a_1u\tau(a_2)\sh v) \right]b_2 \\
  &  -\left[ a_1(u \sh b_1v\tau(b_2))+b_1(a_1 u \sh v\tau(b_2)) \right]a_2 \\
 =& ~ a_1ua_2 \sh b_1 v_1 b_2 - (a_1u\tau(a_2)\sh b_1v)b_2 - (a_1u \sh b_1v\tau(b_2))a_2 \\
 =& ~ a_1ua_2 \tsh b_1 v_1 b_2,  
\end{align*}
which concludes the proof. 
\end{proof}

Now an explicit calculation shows that
 \begin{align*}
   &~up^ay^b \square_{\operatorname{OOZ}} vp^cy^d \\
 = &~ \widetilde\tau \left(z_b z_0^{a-1} \widetilde\tau (u) \ast_{\operatorname{OOZ}}  z_d z_0^{c-1} \widetilde\tau (v) \right) \\
 = &~\widetilde\tau  \left(z_b z_0^{a-1} \widetilde\tau (u) \ast_1 z_d z_0^{c-1} \widetilde\tau (v) - z_b( z_0^{a-1} \widetilde\tau (u) \ast_1 z_{d-1} z_0^{c-1} \widetilde\tau (v)) \right. \\
 & - \left. z_d(z_{b-1} z_0^{a-1} \widetilde\tau (u) \ast_1 z_0^{c-1} \widetilde\tau (v)) -z_{b+d-1}( z_0^{a-1} \widetilde\tau (u) \ast_1 z_0^{c-1} \widetilde\tau (v))\right) \\
 = &~ up^ay^b \square_{1} vp^cy^d - \left(up^{a-1} \square_{1}vp^{c}y^{d-1} \right)py^b - \left(up^{a}y^{b-1} \square_{1}vp^{c-1} \right)py^d - \left(up^{a-1} \square_{1}vp^{c-1} \right)py^{b+d-1}.
\end{align*}
By Theorem \ref{theo:SZdual} the product $\square_1$ coincides with $\sh_1$ which is the same as $\sh$ modulo lower weight terms. This concludes the proof. 
\end{proof}

If we compare the definition of $\ssh$ and Lemma \ref{lem:altstarsh} we see that the product $\ssh$ can be calculated in terms of $\ssh$ itself or in terms of the ordinary shuffle product $\sh$. Duality implies that this fact is also true for the quasi-shuffle product of the OOZ-model. 

\begin{theorem}
 Let $Y:=\{z_k\colon k\in \bZ \}$. The quasi-shuffle-like product $\ast_{\operatorname{OOZ}}\colon \bQ\langle Y\rangle \otimes \bQ\langle Y\rangle \to \bQ\langle Y\rangle$ is given by 
$\be \ast_{\operatorname{OOZ}} w = w \ast_{\operatorname{OOZ}} \be = w$ and 
\begin{align*}
 uz_m \ast_{\operatorname{OOZ}} vz_n =&~ (u \ast_{\operatorname{OOZ}} vz_n)z_m + (uz_m \ast_{\operatorname{OOZ}} v)z_n + (u\ast_{\operatorname{OOZ}} v)z_{n+m} - \delta(v)uz_mz_{n-1} \\ 
 &- \delta(u)vz_nz_{m-1} -\delta(v)uz_{n+m-1} - \delta(u)vz_{n+m-1} + \delta(u)\delta(v)z_{n+m-1}. 
\end{align*}
\end{theorem}

\begin{proof}
 We define 
 \begin{align*}
  uz_m \tast vz_n :=&~ (u \ast_{\operatorname{OOZ}} vz_n)z_m + (uz_m \ast_{\operatorname{OOZ}} v)z_n + (u\ast_{\operatorname{OOZ}} v)z_{n+m} - \delta(v)uz_mz_{n-1} \\ 
 &- \delta(u)vz_nz_{m-1} -\delta(v)uz_{n+m-1} - \delta(u)vz_{n+m-1} + \delta(u)\delta(v)z_{n+m-1}
 \end{align*}
and prove $w\tast w'  =w\ast_{\operatorname{OOZ}}w'$ by induction on the sum of the length of $w$ and $w'$. The length one case is trivial. For length two we observe 
\begin{align*}
 z_m\tast z_n = z_m z_n + z_nz_m + z_{n+m} -z_mz_{n-1} - z_nz_{m-1} - z_{n+m-1} = z_m \ast_{\operatorname{OOZ}} z_n.
\end{align*}
Now we consider two cases. First of all, we have 
\begin{align*}
 z_{a_1} u z_{a_2} \tast z_b 
 =& ~ (z_{a_1}u \tast z_b)z_{a_2} + z_{a_1}uz_{a_2}z_b + z_{a_1}u z_{a_2+b} - z_{a_1}uz_{a_2}z_{b-1} - z_{a_1}uz_{a_2+b-1} \\
 =& ~ z_{a_1}(u\ast_1 z_b)z_{a_2} + z_bz_{a_1}uz_{a_2} + z_{a_1+b}z_{a_2} \\
  & - z_{a_1}(u\ast_1z_{b-1})z_{a_2} - z_bz_{a_1-1}uz_{a_2} - z_{a_1+b-1}uz_{a_2} \\
  & + z_{a_1}uz_{a_2}z_b + z_{a_1} uz_{a_2+b} - z_{a_1}uz_{a_2}z_{b-1} - z_{a_1}uz_{a_2+b-1} \\
 =&~ z_{a_1}(uz_{a_2}\ast_1 z_b) + z_bz_{a_1} uz_{a_2} + z_{a_1+b}uz_{a_2} \\
 & - z_{a_1}(uz_{a_2}\ast_1 z_{b-1}) - z_bz_{a_1-1} uz_{a_2} - z_{a_1+b-1}uz_{a_2} \\
 =&~  z_{a_1} u z_{a_2} \ast_{\operatorname{OOZ}} z_b. 
\end{align*}
Furthermore we observe 
\begin{align*}
 z_{a_1}uz_{a_2} \tast z_{b_1}vz_{b_2} 
 =&~ (z_{a_1}u\tast z_{b_1}uz_{b_2})z_{a_2} + (z_{a_1}uz_{a_2} \tast z_{b_1}v)z_{b_2} + (z_{a_1}u\tast z_{b_1}v)z_{z_{a_2+b_2}} \\
 =&~ z_{a_1}(u\ast_1 z_{b_1}vz_{b_2})z_{a_2} + z_{b_1} (z_{a_1}u\ast_1 vz_{b_1})z_{a_2} + z_{a_1+b_1} (u \ast_1 vz_{b_2})z_{a_2} \\
 & -z_{a_1}(u\ast_1 z_{b_1-1}vz_{b_2})z_{a_2} - z_{b_1} (z_{a_1-1}u\ast_1 vz_{b_1})z_{a_2} - z_{a_1+b_1-1} (u \ast_1 vz_{b_2})z_{a_2} \\
 & +z_{a_1}(uz_{a_2}\ast_1 z_{b_1}v)z_{b_2} + z_{b_1}(z_{a_1}uz_{a_2} \ast_1 v )z_{b_2} + z_{a_1+b_1} (uz_{a_1}\ast_1 v)z_{b_2} \\
 & -z_{a_1}(uz_{a_2}\ast_1 z_{b_1-1}v)z_{b_2} - z_{b_1}(z_{a_1-1}uz_{a_2} \ast_1 v )z_{b_2} - z_{a_1+b_1-1} (uz_{a_1}\ast_1 v)z_{b_2} \\
 & +z_{a_1}(u\ast_1 z_{b_1}v)z_{a_2+b_2} + z_{b_1}(z_{a_1}u\ast_1 v)z_{a_2+b_2} + z_{a_1+b_1}(u\ast v)z_{a_2+b_2} \\
 & -z_{a_1}(u\ast_1 z_{b_1-1}v)z_{a_2+b_2} - z_{b_1}(z_{a_1-1}u\ast_1 v)z_{a_2+b_2} - z_{a_1+b_1-1}(u\ast v)z_{a_2+b_2} \\
 =&~ z_{a_1}(uz_{a_2}\ast_1 z_{b_1}vz_{b_2}) + z_{b_1}(z_{a_1}uz_{a_2}\ast_1 vz_{b_2}) + z_{a_1+b_1}(uz_{a_2}\ast_1vz_{b_2}) \\
 &~ -z_{a_1}(uz_{a_2}\ast_1 z_{b_1-1}vz_{b_2}) - z_{b_1}(z_{a_1-1}uz_{a_2}\ast_1 vz_{b_2}) - z_{a_1+b_1-1}(uz_{a_2}\ast_1vz_{b_2})\\
 =&~  z_{a_1}uz_{a_2} \ast_{\operatorname{OOZ}} z_{b_1}vz_{b_2}, 
\end{align*}
which concludes the proof.
\end{proof}

\bibliographystyle{abbrv}
\bibliography{library}
\end{document}